\documentclass[11pt,reqno]{amsart}

\usepackage{enumerate}   

\usepackage{amssymb}
\usepackage{amsmath}
\usepackage{amsthm}
\usepackage{latexsym}
\usepackage{amsfonts}
\usepackage{mathrsfs}
\usepackage{bbm}
\usepackage{color}
\usepackage{hyperref}
\usepackage[normalem]{ulem}

\newtheorem{thm}{Theorem}[section]
\newtheorem{cor}[thm]{Corollary}
\newtheorem{lem}[thm]{Lemma}
\newtheorem{prop}[thm]{Proposition}

\theoremstyle{definition}
\newtheorem{defn}[thm]{Definition}
\newtheorem{rem}[thm]{Remark}
\newtheorem{ex}[thm]{Example}

\numberwithin{equation}{section}
\allowdisplaybreaks

\newcommand{\ol}{\overline}
\newcommand{\ds}{\dotplus}

\newcommand{\rmref}[1]{{\rm\ref{#1}}}

\newcommand{\one}{\mathbbm{1}}
\newcommand{\braces}[1]{{\rm (}#1{\rm )}}

\newcommand{\C}{\ensuremath{\mathbb C}}    
\newcommand{\N}{\ensuremath{\mathbb N}}    
\newcommand{\D}{\ensuremath{\mathbb D}}    
\newcommand{\T}{\ensuremath{\mathbb T}}    

\newcommand{\hproduct}{(\cdot\,,\cdot)}
\newcommand{\<}{\langle}
\renewcommand{\>}{\rangle}

\newcommand{\calA}{\mathcal A}

\newcommand{\calF}{\mathcal F}         
         
\newcommand{\calH}{\mathcal H}

\newcommand{\calK}{\mathcal K}         
\newcommand{\calL}{\mathcal L}

\newcommand{\la}{\lambda}
\newcommand{\veps}{\varepsilon}
\newcommand{\vphi}{\varphi}

\newcommand{\mat}[4]
{
   \begin{pmatrix}
      #1 & #2\\
      #3 & #4
   \end{pmatrix}
}

\renewcommand{\Re}{\operatorname{Re}}
\newcommand{\linspan}{\operatorname{span}}
\renewcommand{\ker}{\operatorname{ker}}
\newcommand{\ran}{\operatorname{ran}}

\newcommand{\codim}{\operatorname{codim}}

\newcommand{\sess}{\sigma_{\rm ess}}

\newcommand{\Sra}{\Rightarrow}

\definecolor{darkgreen}{rgb}{0,0.6,0.1}

\newcommand{\Id}{\operatorname{Id}}
\newcommand{\id}{\operatorname{id}}
\newcommand{\mult}{\operatorname{mult}}
\newcommand{\ind}{\operatorname{ind}}

\newcommand{\BB}{\mathbb B}

\newcommand{\nul}{\operatorname{nul}}
\newcommand{\defi}{\operatorname{def}}

\begin{document}
\title[]{Dynamical Sampling on Finite Index Sets}

\author[C. Cabrelli]{Carlos Cabrelli}
\author[U. Molter]{Ursula Molter}
\author[V. Paternostro]{Victoria Paternostro}

\address[C. Cabrelli, U. Molter and V. Paternostro]{Universidad de Buenos Aires, Facultad de Ciencias Exactas y Naturales, Departamento de Matem\'atica, Buenos Aires, Argentina, and CONICET-Universidad de Buenos Aires, Instituto de Investigaciones Matem\'aticas Luis A. Santalo (IMAS), Buenos Aires, Argentina}

\email[C. Cabrelli]{cabrelli@dm.uba.ar}
\urladdr[C. Cabrelli]{http://mate.dm.uba.ar/~cabrelli/}
\email[U. Molter]{umolter@dm.uba.ar}
\urladdr[U. Molter]{http://mate.dm.uba.ar/~umolter/}
\email[V. Paternostro]{vpater@dm.uba.ar}

\author[F. Philipp]{Friedrich Philipp}
\address[F. Philipp]{KU Eichst\"att-Ingolstadt, Mathematisch-Geographische Fakult\"at, Ostenstra\ss e 26, Kollegiengeb\"aude I Bau B, 85072 Eichst\"att, Germany}
\email{fmphilipp@gmail.com}
\urladdr{http://www.ku.de/?fmphilipp}

\thanks{The research of the authors was partially supported by UBACyT under grants 20020130100403BA, 20020130100422BA, and 20020150200110BA, by CONICET (PIP 11220110101018), and MinCyT Argentina under grant PICT-2014-1480.}

\begin{abstract}
We consider bounded operators $A$ acting iteratively on a finite set of vectors $\{f_i:i\in I\}$ in a Hilbert space $\calH$ and address the problem of providing necessary and sufficient conditions for the collection of iterates $\{A^nf_i: i\in I, n=0,1,2,....\}$ to form a frame for the space $\calH$. For normal operators $A$ we completely solve the problem by proving a characterization theorem. Our proof incorporates techniques from different areas of mathematics, such as operator theory, spectral theory, harmonic analysis, and complex analysis in the unit disk. In the second part of the paper we drop the strong condition on $A$ to be normal. Despite this quite general setting, we are able to prove a characterization which allows to infer many strong necessary conditions on the operator $A$. For example, $A$ needs to be similar to a contraction of a very special kind. We also prove a characterization theorem for the finite-dimensional case. --- These results provide a theoretical solution to the so-called Dynamical Sampling problem where a signal $f$ that is evolving in time through iterates of an operator $A$ is spatially sub-sampled at various times and one seeks to reconstruct the signal $f$ from these spatial-temporal samples.
\end{abstract}


\maketitle
\thispagestyle{empty}

\section{Introduction}
Given a system of vectors $\{f_i\}_{i\in I}$ from some Hilbert space $\calH$ and a normal operator $A$, we consider  the collection of iterates $\calA=\{A^nf_i: i\in I, n = 1, \dots, l_i\}.$  
We are interested in the special structure of this set. The relevant questions are when the set $\mathcal{A}$ is complete in $\calH$, when it is a basis, when it is a Bessel sequence, or when it forms a frame for $\calH$. In particular, one  seeks conditions on the operator $A$, the vectors $\{f_i\}$ and the number of iterations $l_i$ in order to guarantee the desired properties of the system $\calA$.

These questions are in general of a very difficult nature. Their answers require the use of notions and techniques of different areas of mathematics  such as operator theory, spectral theory, harmonic analysis, and complex analysis in the unit disk. The results are most of the time unexpected. Just to mention some examples, it was proved in \cite{acmt} that if $A$ is a diagonal operator in $\ell^2(\N)$, the collection $\calA$ can never be a basis of $\calH$. It was also shown in \cite{acmt} that for these kinds of operators the orbit $(A^nf)_{n\in\N}$ of one vector $f\in\ell^2(\N)$ is a frame for $\ell^2(\N)$ if and only if the sequence of eigenvalues of $A$ is a set of interpolation for the Hardy space $H^2(\D)$ of the unit disk $\D$ together with some boundedness condition on the vector $f$.
 
In signal processing this problem constitutes an instance of the so-called Dynamical Sampling problem. In Dynamical Sampling a signal $f$ that is evolving in time through an operator $A$ is spatially sub-sampled at multiple times and one seeks to reconstruct the signal $f$ from these spatial-temporal samples, thereby exploiting time evolution (see, e.g., \cite{aadp,accmp,acmt,adk,ap,lv,rclv}). Obviously, the task of reconstructing the signal is an inverse problem. In this paper we give necessary and sufficient conditions on its well-posedness.

In the following, we shall introduce the reader to the motivation, the ideas, and the details of Dynamical Sampling, describe the current state of research, and expose our contribution in this paper.

\vspace*{.2cm}
\subsection{Motivation and idea of Dynamical Sampling}
Let us assume that we are given the task of spatially sampling (i.e., evaluating) a signal $f$ from a function space $\calH$ in such a way that $f$ can later be recovered from these samples. The first idea is, of course, to sample the function $f$ at many convenient positions $x_i$ -- hoping that the knowledge on the properties of the functions in $\calH$ suffices to recover $f$ from the samples $f(x_i)$. However, in real-world scenarios there are typically many restrictions that one has to deal with. For example, the access to some of the required places $x_i$ might be prohibited. Another problem is that sensors are usually very expensive so that the installation of a great number of them in order to guarantee a high-accuracy recovery becomes a crucial financial problem.

However, in many situations the signal $f$ also varies in time and the evolution law is known. The idea of Dynamical Sampling is to avoid the above-mentioned obstacles by reducing the number of positions $x_i$ and to sample $f$ not only at one but at various times, thereby exploiting the knowledge of the evolution law. This idea was for the first time considered by Lu et al. (see \cite{lv,rclv}), where the authors investigated signals obeying the heat equation.

Therefore, in our model let us add a time entry to $f$ and assume that $f(t,\cdot)$ remains in $\calH$ for each $t\ge 0$ and that $f(t,x)$ is a solution to a dynamical system. In the simplest case, where this dynamical system is homogeneous and linear, the function $u(t) = f(t,\cdot)$, $t\ge 0$, maps $[0,\infty)$ to $\calH$ and satisfies $\dot u(t) = Bu(t)$, where $B$ is a generator of a semigroup $(T_t)_{t\ge 0}$ of operators. The solution of this Cauchy problem is then given by $u(t) = T_tu_0$, where $u_0 = u(0)$ is our original signal. If we sample uniformly in time and at fixed positions, the samples are of the following form:
$$
f(nt_0,x_i) = u(nt_0)(x_i) = [T_{t_0}^nu_0](x_i), \qquad n = 0,1,2,\ldots,n_i,\;i\in I.
$$
If $\calH$ is in fact a reproducing kernel Hilbert space (RKHS) with kernel $K$, we have
$$
f(nt_0,x_i) = \left\<T_{t_0}^nu_0,K_{x_i}\right\> = \left\<u_0,A^nK_{x_i}\right\>,
$$
where $A := T_{t_0}^*$. Since the original task was to recover $u_0$ from the retrieved information, the question now becomes: ``{\it Is $(A^nK_{x_i})_{n,i}$ complete in $\calH$?}''. If one requires the recovery to be a stable process, the question is ``{\it Is $(A^nK_{x_i})_{n,i}$ a frame for $\calH$?}''.

In the general Dynamical Sampling problem (see, e.g., \cite{acmt,accmp}), the $K_{x_i}$ in a RKHS are replaced by vectors $f_i$ from an arbitrary Hilbert space $\calH$. The question is now the following:

\begin{quote}
{\it For which operators $A$, which sets $I,N_i\subset\N$, and which vectors $f_i\in\calH$, $i\in I$, is the system $(A^nf_i)_{i\in I,\,n\in N_i}$ complete in $\calH$ or a frame for $\calH$?}
\end{quote}

\noindent In this paper, we focus on the case where the iteration sets $N_i$ do not depend on $i\in I$ and equal $N_i := N := \{0,1,\ldots,\dim\calH - 1\}$. In particular, $N = \N$ if $\calH$ is infinite-dimensional.

\vspace*{.2cm}
\subsection{Previous works on the topic and our contribution}
The history of Dynamical Sampling is fairly young. The papers \cite{lv,rclv} of Vetterli et al.\ can be seen as the first works on Dynamical Sampling. They consider the sampling of signals under diffusion evolution. The next series of papers, written by Aldroubi et al.\ (see, e.g., \cite{aadp,adk}), was dealing with the special type of convolution operators $A$. The first paper on the above-mentioned problem in its most general form was \cite{acmt}, in which the authors considered both the finite-dimensional and the infinite-dimensional case. They proved that if $A\in\C^{d\times d}$ is diagonalizable, then $(A^nf_i)_{i\in I,\,n\in N}$ is a frame for $\calH = \C^d$ if and only if for each eigenprojection $P$ of $A$ we have that $(Pf_i)_{i\in I}$ is complete in $P\calH$. If the operator $A$ is not diagonalizable, the above statement can be generalized, using the Jordan canonical form: If $A\in\C^{d\times d}$ and $\calF = \{f_i : i\in I\}$, then $(A^nf_i)_{i\in I,\,n\in N}$ is a frame for $\calH = \C^d$ if and only if for each eigenvalue $\la$ the projection $Q_\la$ of $\calF$ onto the cyclic Jordan vectors for the eigenvalue $\la$ along the image of $A-\la$ is complete in $Q_\la\calH$. The drawback of that approach is that it practically requires the knowledge of the entire Jordan structure of $A$. Here, we provide another necessary and sufficient condition which is easier to check (cf.\ Theorem \ref{t:main_findim}). In fact, the projection $Q_\la$ from above can be replaced by {\em any} projection onto a complementary subspace of the image of $A-\la$. Hence, $Q_\la$ can be replaced by the orthogonal projection onto $\ker(A^*-\ol\la)$.

Concerning the infinite-dimensional situation, the most interesting result in \cite{acmt} addresses the one-vector problem (i.e., $|I|=1$). This result was further improved in \cite{accmp} and \cite{ap}. Its final version reads as follows: {\it If $A$ is a normal operator, the system $(A^nf)_{n\in\N}$ is a frame for $\calH$ if and only if {\rm (a)} $A = \sum_{in\N}\la_j\<\,\cdot\,,e_j\>e_j$ with an ONB $(e_j)_{j\in\N}$, {\rm (b)} the sequence $(\la_j)_{j\in\N}$ is uniformly separated in the unit disk \braces{cf.\ page \pageref{p:us}}, and {\rm (c)} the sequence $(|\<f,e_j\>|^2/(1-|\la_j|^2)_{j\in\N}$ is bounded from below and above.} One of the aims of this paper is to generalize this result to arbitrary finite index sets $I$. However, this problem turns out to be more difficult to tackle than one might think at first glance -- the attempt of using the same techniques as in the case $|I|=1$ terribly fails. Nevertheless, we find the right methods to deal with the new situation (see Theorem \ref{t:characterization}). Three conditions in our characterization are generalizations of the conditions (a)--(c) above in the one-vector case. But indeed one has to add a fourth condition which is trivially satisfied when $|I|=1$.

Very little is known on the Dynamical Sampling problem for general non-normal bounded operators $A$. It was only proved in \cite{ap} that for $(A^nf_i)_{n\in\N,\,i\in I}$ to be a frame for $\calH$ it is necessary that $A^*$ be strongly stable, i.e., $(A^*)^nf\to 0$ as $n\to\infty$ for each $f\in\calH$. Here, we complete this condition to a characterizing set of three conditions (cf.\ Theorem \ref{t:charac}). Using this theorem, we completely characterize the class of all operators $A$ for which there exists some finite set $\{f_i : i\in I\}$ such that $(A^nf_i)_{n\in\N,\,i\in I}$ is a frame for $\calH$. In fact, these are the operators that are similar to a strongly stable contraction $T$ for which $\Id - TT^*$ is of finite rank. We also characterize the Riesz bases of the form $(A^nf_i)_{n\in\N,\,i\in I}$ when $I$ is finite. In this case, the operator $A$ has to be similar to the $|I|$-th power of the unilateral shift in $\ell^2(\N)$.

\vspace*{.2cm}
\subsection{Outline}
The present paper is organized as follows. Section \ref{s:normal} contains our main results concerning Dynamical Sampling on finite index sets with normal operators, including the above-mentioned characterization consisting of four conditions. In Section \ref{s:general} we drop the requirement that $A$ be normal and provide our results, summarized above, for this much more general setting. In Section \ref{s:findim} we deal with the finite-dimensional situation and prove a characterization result in which the condition can be very easily checked.

\vspace*{.2cm}
\subsection{Notation}
We conclude this Introduction by fixing the notation that we shall use throughout this paper. By $\N$ we denote the set of the natural numbers {\it including zero}. Unit circle and open unit disk in $\C$ are denoted by $\T$ and $\D$, respectively, i.e.,
$$
\T = \{z\in\C : |z|=1\}\qquad\text{and}\qquad\D = \{z\in\C : |z| < 1\}.
$$
The $p$-th Hardy space on the unit disk, $1\le p\le\infty$, is denoted by $H^p(\D)$. Recall that especially $H^2(\D)$ consists of those functions that have a representation $\vphi(z) = \sum_{n=0}^\infty c_nz^n$, $z\in\D$, where $c = (c_n)_{n\in\N}\in\ell^2(\N)$, and that $\|\vphi\|_{H^2(\D)} = \|c\|_2$.

Throughout, $\calH$ stands for a separable Hilbert space. If $\calK$ is another Hilbert space, by $L(\calH,\calK)$ we denote the set of all bounded linear operators from $\calH$ to $\calK$ which are defined on all of $\calH$. As usual, we set $L(\calH) := L(\calH,\calH)$. The kernel (i.e., the null-space) and the range (i.e., the image) of $T\in L(\calH)$ are denoted by $\ker T$ and $\ran T$, respectively.

\vspace{.2cm}
\section{Dynamical Sampling with normal operators}\label{s:normal}
In this section we investigate sequences of the form $(A^nf_i)_{n\in\N,\,i\in I}$ where $A$ is a bounded normal operator in $\calH$, $I$ an at most countable index set, and $(f_i)_{i\in I}\subset\calH$. The spectral measure of $A$ will be denoted by $E$. Throughout, we set
\begin{equation}\label{e:calA}
\calA := \calA(A,(f_i)_{i\in I}) := (A^nf_i)_{n\in\N,\,i\in I}.
\end{equation}
In the sequel, we will often be dealing with diagonal operators -- a special class of normal operators -- which we define as follows.

\begin{defn}
A {\em diagonal operator} in $\calH$ is of the form $A = \sum_{j\in J}\la_jP_j$ (the series converging in the strong operator topology), where $J$ is a finite or countable index set, $(\la_j)_{j\in J}\subset\C$ a bounded sequence of scalars, and $(P_j)_{j\in J}$ a sequence of orthogonal projections with $P_jP_k = 0$ for $j\neq k$. The series $\sum_{j\in J}\la_jP_j$ is called a {\em normal form} of $A$ if $\la_j\neq\la_k$ for $j\neq k$ and $\sum_{j\in J}P_j = \Id$. The {\em multiplicity} of a diagonal operator $A$ is defined by
$$
\mult(A) := \max\{\dim P_j\calH : j\in J\},
$$
where $(P_j)_{j\in J}$ is the sequence of orthogonal projections in a normal form of $A$. If the maximum should not exist, we set $\mult(A) := \infty$ and say that $A$ has infinite multiplicity.
\end{defn}

The normal form of a diagonal operator is obviously unique up to permutations of $J$. Moreover, it is clear that the $\la_j$ in the normal form of $A$ are the distinct eigenvalues of $A$ and that $P_j$ projects onto the eigenspace $\ker(A - \la_j)$. Note that every diagonal operator is bounded and normal. If $f\in\calH$ and $E$ denotes the spectral measure of $A$, the measure $\mu_f := \|E(\cdot)f\|^2$ obviously takes the form
\begin{equation}\label{e:mu_f}
\mu_f = \sum_{j\in J}\delta_{\la_j}\|P_jf\|^2.
\end{equation}
Recall that the {\em pseudo-hyperbolic metric} $\varrho$ on the open unit disk is defined by
\begin{equation*}
\varrho(z,w) := \left|\frac{z-w}{1-z\ol w}\right|,\quad z,w\in\D.
\end{equation*}
Since
\begin{equation}\label{e:hyp_identity}
|1-z\ol w|^2 = |z-w|^2 + (1-|z|^2)(1-|w|^2),
\end{equation}
we always have $\varrho(z,w) < 1$. It is well known that $\varrho$ is indeed a metric on $\D$. For $z\in\D$ and $r > 0$ by $\BB_r(z)$ we denote the pseudo-hyperbolic ball ($\varrho$-ball) in $\D$ of radius $r$ and center $z$, i.e.,
$$
\BB_r(z) = \{\la\in\D : \varrho(\la,z) < r\}.
$$
Note that $\BB_r(z) = \{\la\in\D : |\la-z'| < r'\}$ with certain $r'< r$ and $z' = tz$, where $t < 1$.

A sequence $\Lambda = (\la_j)_{j\in\N}$ in the open unit disk $\D$ is called {\em separated} if
$$
\inf_{j\neq k}\varrho(\la_j,\la_k) > 0.
$$
The sequence $\Lambda$ is called {\em uniformly separated} if\label{p:us}
$$
\inf_{n\in\N}\,\prod_{j\neq k}\varrho(\la_j,\la_k)\,>\,0.
$$
Obviously, a uniformly separated sequence is separated. We refer to Appendix \ref{a:sequences} for more detailed relationships between these notions. The next theorem was proved in \cite[Theorem 3.14]{acmt}.

\begin{thm}\label{t:one}
Let $A = \sum_{j=0}^\infty\la_jP_j$ be a diagonal operator in normal form and $f\in\calH$. Then $(A^nf)_{n\in\N}$ is a frame for $\calH$ if and only if the following statements hold:
\begin{enumerate}
\item[{\rm (i)}]   $(\la_j)_{j\in\N}$ is a uniformly separated sequence in $\D$.
\item[{\rm (ii)}]  $\dim P_j\calH = 1$ for all $j\in\N$.
\item[{\rm (iii)}] There exist $\alpha,\beta > 0$ such that
$$
\alpha\,\le\,\frac{\|P_jf\|^2}{1 - |\la_j|^2}\,\le\,\beta\quad\text{for all $j\in\N$}.
$$
\end{enumerate}
\end{thm}

Note that the system $(A^nf)_{n\in\N}$ in Theorem \ref{t:one} corresponds to systems of the form $\calA$ in \eqref{e:calA} with the index set $I$ being a singleton, i.e., $|I|=1$. In this section it is our aim to generalize Theorem \ref{t:one} to arbitrary finite index sets $I$ (see Theorem \ref{t:characterization} below). Although this might seem to be a trivial task, our treatment shows that this is not the case.

In order to formulate our main result Theorem \ref{t:characterization}, it is necessary to introduce a few more notions concerning sequences in the unit disk. For a sequence $\Lambda = (\la_j)_{j\in\N}\subset\D$ in the open unit disk we agree to write
\begin{equation}\label{e:veps}
\veps_j := \sqrt{1 - |\la_j|^2},\qquad j\in\N.
\end{equation}
Although $\veps_j$ depends on $\Lambda$, it will always be clear which $\Lambda$ it refers to. We also define the linear {\em evaluation operator}
\begin{equation}\label{e:eval1}
T_\Lambda : H^2(\D)\supset D(T_\Lambda)\to\ell^2(\N),\qquad T_\Lambda\vphi := (\veps_j\vphi(\la_j))_{j\in\N},
\end{equation}
on its natural domain
\begin{equation}\label{e:eval2}
D(T_\Lambda) := \left\{\vphi\in H^2(\D) : (\veps_j\vphi(\la_j))_{j\in\N}\in\ell^2(\N)\right\}.
\end{equation}
Note that $\veps_j\vphi(\la_j) = \<\vphi,K_{\la_j}\>_{H^2(\D)}$, where
\begin{equation}\label{e:K}
K_\la(z) = \sqrt{1 - |\la|^2}\sum_{n=0}^\infty\ol\la^nz^n = \frac{\sqrt{1 - |\la|^2}}{1 - \ol\la z},\qquad\la,z\in\D,
\end{equation}
is the normalized reproducing kernel of $H^2(\D)$. Hence, the operator $T_\Lambda$ is the analysis operator corresponding to the sequence $(K_{\la_j})_{j\in\N}$. As every analysis operator is closed on its natural domain, it follows from the closed graph theorem that $T_\Lambda$ is a bounded operator from $H^2(\D)$ to $\ell^2(\N)$ if and only if $D(T_\Lambda) = H^2(\D)$.

The next theorem is the main result in this section. For a Bessel sequence $E$ in $\calH$ we let $C_E$ denote the analysis operator of $E$.

\begin{thm}\label{t:characterization}
If $|I|$ is finite, then the system $\calA = (A^nf_i)_{n\in\N,\,i\in I}$ is a frame for $\calH$ if and only if the following conditions are satisfied:
\begin{enumerate}
\item[{\rm (i)}]   $A = \sum_{j=0}^\infty\la_jP_j$ is a diagonal operator \braces{in normal form} having multiplicity $\mult(A)\le |I|$.
\item[{\rm (ii)}]  $\Lambda = (\la_j)_{j\in\N}$ is a union of $|I|$ uniformly separated sequences in $\D$.
\item[{\rm (iii)}] There exist $\alpha,\beta > 0$ such that for all $j\in\N$ and all $h\in P_j\calH$ we have
\begin{equation}\label{e:fs}
\alpha(1-|\la_j|^2)\|h\|^2\,\le\,\sum_{i\in I}|\<h,P_jf_i\>|^2\,\le\,\beta(1-|\la_j|^2)\|h\|^2.
\end{equation}
\item[{\rm (iv)}]  $(\ran T_{\Lambda})^{|I|} + \ker C_E^* = \ell^2(I\times\N)$, where $E := ((1-|\la_j|^2)^{-1/2}P_jf_i)_{j\in\N,\,i\in I}$.
\end{enumerate}
\end{thm}

Before we head towards the proof of Theorem \ref{t:characterization}, let us first make a few remarks.

\begin{rem}\label{r:jojo}
(a) The necessity of condition (i) in Theorem \ref{t:characterization} for $\calA$ to be a frame for $\calH$ was already proved in \cite{ap}.

\smallskip
(b) Condition (iii) means that for each $j\in\N$ the finite system $((1-|\la_j|^2)^{-1/2}P_jf_i)_{i\in I}$ is a frame for $P_j\calH$ with frame bounds $\alpha$ and $\beta$. Since the frame bounds are independent of $j\in\N$, condition (iii) is equivalent to saying that the system $E = ((1-|\la_j|^2)^{-1/2}P_jf_i)_{j\in\N,\,i\in I}$ is a frame for $\calH$ with frame bounds $\alpha$ and $\beta$.

\smallskip
(c) Here and in the sequel, we will make use of the following notion:
\begin{equation}\label{e:TLaI}
T_{\Lambda,I} := \bigoplus_{i\in I}T_\Lambda.
\end{equation}
That is, $T_{\Lambda,I}$ is a closed linear operator mapping from $D(T_{\Lambda,I}) = D(T_\Lambda)^{|I|}\subset H_I^2$ (here, $H_I^2 = (H^2(\D))^{|I|}$) to $(\ell^2(\N))^{|I|} = \ell^2(I\times\N)$. Hence, $(\ran T_\Lambda)^{|I|} = \ran T_{\Lambda,I}$. Since $E$ is a frame for $\calH$, the relation $\ran T_{\Lambda,I} + \ker C_E^* = \ell^2(I\times\N)$ in (iv) can be equivalently replaced by
$$
C_E^*\ran T_{\Lambda,I} = \calH.
$$
Indeed, assume that the relation in (iv) holds. As $E$ is a frame for $\calH$, for any $h\in\calH$ there exists $c\in\ell^2(I\times\N)$ such that $C_E^*c = h$. Now, $c = c_1+c_2$ with $c_1\in\ran T_{\Lambda,I}$ and $c_2\in\ker C_E^*$. Hence, $h = C_E^*c_1\in C_E^*\ran T_{\Lambda,I}$. Conversely, if $C_E^*\ran T_{\Lambda,I} = \calH$ and $c\in\ell^2(I\times\N)$, then $C_E^*c = C_E^*T_{\Lambda,I}h$ for some $h\in\calH$. Thus, $c - T_{\Lambda,I}h\in\ker C_E^*$, i.e., $c\in\ran T_{\Lambda,I} + \ker C_E^*$.

\smallskip
(d) If (ii) holds and $\la_i\neq\la_j$ for $i\neq j$ (which follows from (i)), then $(\ran T_{\Lambda})^{|I|}$ in (iv) is dense in $\ell^2(I\times\N)$. Indeed, since (ii) implies that the operator $T_\Lambda$ is bounded and everywhere defined on $H^2(\D)$ (cf.\ Theorem \ref{t:finite_union}), the claim follows from Lemma \ref{l:dense}.

\smallskip
(e) Note that (ii) does not prevent $\Lambda$ to be a union of less than $|I|$ uniformly separated sequences because each subsequence of a uniformly separated sequence is also uniformly separated. Hence, $\Lambda$ might even be uniformly separated itself. In this case, we know that $\ran T_\Lambda = \ell^2(\N)$ (see Theorem \ref{t:interpolating}), so that condition (iv) is trivially satisfied.

\smallskip
(f) As noted in the last remark, (iv) follows from (ii) if $\Lambda$ is uniformly separated. In particular, (iv) is not necessary to state in the case $|I|=1$ (cf.\ Theorem \ref{t:one}). However, if $|I|>1$, condition (iv) does in general not follow from (i)--(iii). As an example, choose a sequence $\Lambda = (\la_j)_{j\in\N}$ which is a union of no more than $|I|$ uniformly separated sequences, but is not uniformly separated itself. In addition, choose orthogonal projections $P_j$ such that $\sum_{j=0}^\infty P_j = \Id$ (in the strong sense) and $\dim P_j\calH = |I|$ for each $j\in\N$ as well as orthonormal bases $(g_{ij})_{i\in I}$ for $P_j\calH$, $j\in\N$. Now, define $A := \sum_{j=0}^\infty\la_j P_j$ and $f_i := \sum_{j=0}^\infty\veps_jg_{ij}$, $i\in I$, where $\veps_j = \sqrt{1-|\la_j|^2}$. Then conditions (i)--(iii) in Theorem \ref{t:characterization} are satisfied, but (iv) is not as $E = (\veps_j^{-1}P_jf_i)_{j\in\N,\,i\in I} = (g_{ij})_{j\in\N,\,i\in I}$ is an orthonormal basis of $\calH$ (and thus $\ker C_E^* = \{0\}$) and $\Lambda$ is not uniformly separated (i.e., $\ran T_\Lambda\neq\ell^2(\N)$).
\end{rem}

\
\\
For an at most countable index set $I$ we define $H^2_I := \bigoplus_{i\in I}H^2(\D)$. This is the space of tuples $\phi = (\vphi_i)_{i\in I}$, where $\vphi_i\in H^2(\D)$ for each $i\in I$, such that $\sum_{i\in I}\|\vphi_i\|_{H^2(\D)}^2 < \infty$. One defines $\|\phi\|_{H^2_I} := (\sum_{i\in I}\|\vphi_i\|_{H^2(\D)}^2)^{1/2}$. In addition, we shall write the tensor product of a sequence $\vec y = (y_i)_{i\in I}\subset\C$ and a function $\vphi\in H^2(\D)$ as $\vec y\vphi$ (i.e., $(\vec y\vphi)(z) = (y_i\vphi(z))_{i\in I}$, $z\in\D$). The result is an element of $H^2_I$ if and only if $\vec y\in\ell^2(I)$. In this case,
\begin{equation}\label{e:prodnorm}
\|\vec y\vphi\|_{H^2_I} = \|\vec y\|_2\|\vphi\|_{H^2(\D)}.
\end{equation}

The following theorem will be used in the proof of Theorem \ref{t:characterization}. However, it might be of independent interest. Here, the index set $I$ is allowed to be infinite.

\begin{thm}\label{t:riesz}
Let $A = \sum_{j=0}^\infty\la_jP_j$ be a diagonal operator in normal form with $(\la_j)_{j\in\N}\subset\D$ and let $f_i\in\calH$, $i\in I$. For $j\in\N$ we put $n_j := \dim P_j\calH$ \braces{where possibly $n_j = \infty$} and $L_j := \{1,\ldots,n_j\}$. Moreover, let $(e_{jl})_{l=1}^{n_j}$ be an orthonormal basis of $P_j\calH$ and for $l\in L_j$ define $\vec y_{jl} := \veps_j^{-1}(\<e_{jl},f_i\>)_{i\in I}$ as well as $\phi_{jl} := \vec y_{jl}K_{\la_j}$. Then the following statements hold.
\begin{enumerate}
\item[{\rm (i)}]  $\calA$ is a Bessel sequence in $\calH$ if and only if $\vec y_{jl}\in\ell^2(I)$ \braces{i.e., $\phi_{jl}\in H_I^2$} for each $j\in\N$ and each $l\in L_j$ and $\Phi = (\phi_{jl})_{j\in\N,\,l\in L_j}$ is a Bessel sequence in $H_I^2$. In this case, the Bessel bounds of $\calA$ and $\Phi$ coincide.
\item[{\rm (ii)}] $\calA$ is a frame for $\calH$ if and only if $\vec y_{jl}\in\ell^2(I)$ for each $j\in\N$ and each $l\in L_j$ and $\Phi = (\phi_{jl})_{j\in\N,\,l\in L_j}$ is a Riesz sequence in $H_I^2$. In this case, the frame bounds of $\calA$ coincide with the Riesz bounds of $\Phi$.
\end{enumerate}
\end{thm}
\begin{proof}
In the following we will often make use of the unitary operator $U : \ell^2(I\times\N)\to H_I^2$, defined by $Ux = (\vphi_i)_{i\in I}$, where $\vphi_i(z) = \sum_{n\in\N}x_{in}z^n$, $z\in\D$.

Assume that $\calA$ is a Bessel sequence in $\calH$ with Bessel bound $\beta > 0$. Then for $h\in\calH$ we have that
$$
\sum_{n=0}^\infty\sum_{i\in I}|\<h,A^nf_i\>|^2\,\le\,\beta\|h\|_2^2.
$$
Let $j\in\N$ and $l\in L_j$. Since for $h = e_{jl}$ we have
$$
\sum_{n=0}^\infty\sum_{i\in I}|\<h,A^nf_i\>|^2 = \sum_{n=0}^\infty\sum_{i\in I}|\la_j|^{2n}|\<e_{jl},f_i\>|^2 = \sum_{i\in I}\frac{|\<e_{jl},f_i\>|^2}{1-|\la_j|^2},
$$
it follows that $\sum_{i\in I}\veps_j^{-2}|\<e_{jl},f_i\>|^2\le\beta$, i.e., $\vec y_{jl}\in\ell^2(I)$. Let $\psi = (\psi_i)_{i\in I}\in H^2_I$ and put $x := U^{-1}\psi\in\ell^2(I\times\N)$. Denote the synthesis operator of $\calA$ by $T$. Then for each $j\in\N$ and $l\in L_j$ we have
\begin{align*}
\<Tx,e_{jl}\>
&=\left\<\sum_{i\in I}\sum_{n=0}^\infty x_{in}A^nf_i,e_{jl}\right\> = \sum_{i\in I}\sum_{n=0}^\infty x_{in}\la_j^n\<f_i,e_{jl}\>\\
&= \sum_{i\in I}\psi_i(\la_j)\<f_i,e_{jl}\> = \sum_{i\in I}\frac{\<f_i,e_{jl}\>}{\veps_j}\left\<\psi_i,K_{\la_j}\right\>\\
&= \sum_{i\in I}\left\<\psi_i,\frac{\<e_{jl},f_i\>}{\veps_j}K_{\la_j}\right\> = \<\psi,\phi_{jl}\>.
\end{align*}
Thus, $\sum_{j=0}^\infty\sum_{l=1}^{n_j}|\<\psi,\phi_{jl}\>|^2 = \sum_{j=0}^\infty\sum_{l=1}^{n_j}|\<Tx,e_{jl}\>|^2 = \|Tx\|^2$, which implies that the sequence $(\phi_{jl})_{j\in\N,\,l\in L_j}$ is a Bessel sequence in $H^2_I$. Let $C$ denote its analysis operator. Then the above relation shows that $T = CU$. In particular, the Bessel bounds of both sequences coincide. Moreover, if $\calA$ is a frame for $\calH$, then $C = TU^*$ is onto, meaning that $(\phi_{jl})_{j\in\N,\,l\in L_j}$ is indeed a Riesz sequence and that its lower Riesz bound coincides with the lower frame bound of $(A^nf_i)_{n\in\N,\,i\in I}$.

Assume conversely that $\vec y_{jl}\in \ell^2(I)$ for each $j\in\N$ and $l\in L_j$ and that $(\phi_{jl})_{j\in\N,\,l\in L_j}$ is a Bessel sequence in $H_I^2$ with Bessel bound $\beta > 0$. If $x\in\ell^2(I\times\N)$ with only finitely many non-zero entries and $\psi = Ux$, then $\<Tx,e_{jl}\> = \<\psi,\phi_{jl}\>$, that is, $\|Tx\|^2 = \sum_{j=0}^\infty\sum_{l=1}^{n_j}|\<\psi,\phi_{jl}\>|^2\le\beta\|\psi\|^2 = \beta\|x\|^2$. Thus, $\calA$ is a Bessel sequence. If, in addition, $(\phi_{jl})_{j\in\N,\,l\in L_j}$ is a Riesz sequence, then $T = CU$ is onto, which means that $\calA$ is a frame for $\calH$.
\end{proof}

If the index set $I$ only contains one element, the system $\calA$ has the form $(A^nf)_{n\in\N}$, where $f\in\calH$. In the next lemma we formulate a characterization from \cite{p} for the case of a diagonal operator.

\begin{lem}\label{l:bessel}
Let $A = \sum_{j=0}^\infty\la_jP_j$ be a diagonal operator in normal form such that $\la_j\in\D$ for all $j\in\N$ and $f\in\calH$. Then the following statements are equivalent.
\begin{enumerate}
\item[{\rm (i)}]  The sequence $(A^nf)_{n\in\N}$ is a Bessel sequence in $\calH$.
\item[{\rm (ii)}] There exists a constant $C > 0$ such that
$$
\sum_{j=0}^\infty\left|\vphi(\la_j)\right|^2\|P_jf\|^2\,\le\,C\|\vphi\|_{H^2(\D)}^2\qquad\text{for all }\vphi\in H^2(\D).
$$
\end{enumerate}
\end{lem}
\begin{proof}
Since $A$ is a diagonal operator, we have that $\mu_f = \sum_{j=0}^\infty\delta_{\la_j}\|P_jf\|^2$ (see \eqref{e:mu_f}). Therefore, for every measurable function $\vphi : \C\to\C$ we have
$$
\int|\vphi|^2\,d\mu_f = \sum_{j=0}^\infty|\vphi(\la_j)|^2\|P_jf\|^2.
$$
Hence, (ii) exactly means that $H^2(\D)$ is continuously embedded in $L^2(\mu_f)$. By \cite[Theorem 4.3]{p}, the latter is equivalent to (i).
\end{proof}

In order to prove Theorem \ref{t:characterization} we need one more definition.

\begin{defn}\label{d:index}
Let $\Lambda = (\la_j)_{j\in\N}$ be a sequence in $\D$. If $\Lambda$ is not a finite union of separated sequences, we set $\ind(\Lambda) := \infty$. Otherwise, we define
\begin{equation*}
\ind(\Lambda) := \min\left\{n\in\N : \N = \bigcup_{k=1}^n J_k,\;(\la_j)_{j\in J_k}\text{ is separated for each $k=1,\ldots,n$}\right\}.
\end{equation*}
The value $\ind(\Lambda)$ will be called the {\em index} of the sequence $\Lambda$.
\end{defn}

\begin{proof}[Proof of Theorem \rmref{t:characterization}]
Suppose that $\calA$ is a frame for $\calH$. By \cite[Corollary 1]{ap}, $A$ is a diagonal operator with multiplicity $\mult(A)\le|I|$ having all its eigenvalues in $\D$. Let $A = \sum_{j=0}^\infty\la_jP_j$ be its normal form as in (i) and let $\alpha,\beta > 0$ be the frame bounds of $\calA$. That is,
\begin{equation}\label{e:frame}
\alpha\|h\|^2\,\le\,\sum_{i\in I}\sum_{n=0}^\infty|\<h,A^nf_i\>|^2\,\le\,\beta\|h\|^2,\qquad h\in\calH.
\end{equation}
Fix $j\in\N$. If $h\in P_j\calH$, then $A^*h = \ol{\la_j}h$ and hence we have $|\<h,A^nf_i\>| = |\<(A^*)^nh,f_i\>| = |\la_j|^n|\<h,f_i\>|$. Therefore,
$$
\sum_{n=0}^\infty\sum_{i\in I}|\<h,A^nf_i\>|^2 = \sum_{n=0}^\infty\sum_{i\in I}|\la_j|^{2n}|\<h,f_i\>|^2 = \frac{\sum_{i\in I}|\<h,f_i\>|^2}{1-|\la_j|^2}.
$$
Together with \eqref{e:frame}, this proves (iii). From (iii) we moreover conclude that
\begin{equation}
\sum_{i\in I}\|P_jf_i\|^2\,\ge\,\alpha\veps_j^2(\dim P_j\calH)\,\ge\,\alpha\veps_j^2
\end{equation}
for each $j\in\N$, where (cf.\ \eqref{e:veps}) $\veps_j := \sqrt{1-|\la_j|^2}$, $j\in\N$. For this, simply choose an orthonormal basis of $P_j\calH$ and plug its vectors into $h$ in \eqref{e:fs}. Thus, for $\vphi\in H^2(\D)$ we obtain for the evaluation operator $T_\Lambda$ from \eqref{e:eval1}--\eqref{e:eval2} that
\begin{equation}
\|T_\Lambda\vphi\|_2^2 = \sum_{j=0}^\infty\veps_j^2|\vphi(\la_j)|^2\,\le\,\alpha^{-1}\sum_{i\in I}\sum_{j=0}^\infty|\vphi(\la_j)|^2\|P_jf_i\|^2.
\end{equation}
By Lemma \ref{l:bessel}, the latter expression is bounded from above by $C\|\vphi\|_{H^2(\D)}^2$, where $C$ is some positive constant. Thus, the operator $T_\Lambda$ is everywhere defined and bounded. Due to Theorem \ref{t:finite_union} this means that the sequence $\Lambda$ is a finite union of uniformly separated sequences. To prove (iv), we start by noticing that $T_{\Lambda,I}$ is bounded as $T_\Lambda$ is. Let $h\in\calH$ be arbitrary. Then there exists $(c_{in})_{i\in I,\,n\in\N}\in\ell^2(I\times\N)$ such that $h = \sum_{i\in I}\sum_{n=0}^\infty c_{in}A^nf_i$. Define $\vphi_i\in H^2(\D)$ by $\vphi_i(z) = \sum_{n=0}^\infty c_{in}z^n$ for $i\in I$ and put $\phi := (\vphi_i)_{i\in I}\in H_I^2$. Then
$$
h = \sum_{i\in I}\sum_{j=0}^\infty P_j\sum_{n=0}^\infty c_{in}A^nf_i = \sum_{i\in I}\sum_{j=0}^\infty\sum_{n=0}^\infty c_{in}\la_j^nP_jf_i = \sum_{i\in I}\sum_{j=0}^\infty\veps_j\vphi_i(\la_j)(\veps_j^{-1}P_jf_i),
$$
and thus $h = C_E^*(T_\Lambda\vphi_i)_{i\in I} = C_E^*T_{\Lambda,I}\phi$, where $T_{\Lambda,I}$ is the operator in \eqref{e:TLaI}. Hence, we have $C_E^*\ran T_{\Lambda,I} = \calH$, which implies (iv), see Remark \ref{r:jojo}(c).

It remains to complete the proof of (ii), i.e., showing that $\Lambda$ is a union of $m:=|I|$ (or less) uniformly separated sequences. Taking in to account Theorem \ref{t:durenschuster}, it is sufficient to separate $\Lambda$ into $m$ separated sequences, that is, to show that $\ind(\Lambda)\le m$. For this, we fix some positive number $r < \sqrt{\alpha(8\beta|I|)^{-1}}$ and prove that every pseudo-hyperbolic ball $\BB_r(z)$, $z\in\D$, contains at most $m$ elements of the sequence $\Lambda$. Then the claim follows from Lemma \ref{l:index}. Towards a contradiction, suppose that some ball $\mathbb B_r(z_0)$ contains $m+1$ elements of $\Lambda$. Without loss of generality, we may assume that these elements are $\la_1,\ldots,\la_{m+1}$. Since $\varrho$ is a metric on $\D$, we have $\varrho(\la_j,\la_k) < 2r$ for all $j,k=1,\ldots,m+1$. 

Using the notation of Theorem \ref{t:riesz}, let $\vec y_j := \vec y_{j,1} = \veps_j^{-1}(\<e_{j,1},f_i\>)_{i\in I}\in\C^m$, $j\in\N$. Since $\vec y_1,\ldots,\vec y_{m+1}$ are $m+1$ vectors in $\C^m$, there exists some $c\in\C^{m+1}$ such that $\|c\|_2 = 1$ and $\sum_{j=1}^{m+1}c_j\vec y_j = 0$. By Theorem \ref{t:riesz}, $(\vec y_jK_{\la_j})_{j\in\N}$ is a Riesz sequence with Riesz bounds $\alpha$ and $\beta$. In particular, we have $\|\vec y_j\|_2^2\le\beta$ for all $j\in\N$. Using this, Cauchy-Schwarz, and Lemma \ref{l:lipschitz}, we obtain
\begin{align*}
\alpha
&\le\Bigg\|\sum_{j=1}^{m+1}c_j\vec y_jK_{\la_j}\Bigg\|_{H_I^2}^2
= \Bigg\|\sum_{j=1}^{m}c_j\vec y_j\left(K_{\la_j} - K_{\la_{m+1}}\right)\Bigg\|_{H_I^2}^2\\
&\le \sum_{j=1}^{m}\left\|\vec y_j(K_{\la_j} - K_{\la_{m+1})}\right\|_{H_I^2}^2\le 2\beta\sum_{j=1}^{m}\varrho(\la_j,\la_{m+1})^2\,\le\,8\beta mr^2 < \alpha,
\end{align*}
which is the desired contradiction. Here, we used that $\|\vec y\vphi\|_{H^2_I} = \|\vec y\|_2\|\vphi\|_{H^2(\D)}$ for $\vec y\in\C^m$ and $\vphi\in H^2(\D)$, cf.\ \eqref{e:prodnorm}.

Conversely, let the conditions (i)--(iv) be satisfied. Let us first prove that for each $i\in I$ the system $(A^nf_i)_{n\in\N}$ is a Bessel sequence. For this, fix $i\in I$ and deduce from (iii) that $\|P_jf_i\|^2\le\beta\veps_j^2$ holds for each $j\in\N$. Note that the evaluation operator $T_\Lambda$ from \eqref{e:eval1}--\eqref{e:eval2} is everywhere defined and bounded by (ii) and Theorem \ref{t:finite_union}. Thus, for every $\vphi\in H^2(\D)$ we have
\begin{equation*}
\sum_{j=0}^\infty|\vphi(\la_j)|^2\|P_jf_i\|^2\,\le\,\beta\sum_{j=0}^\infty|\veps_j\vphi(\la_j)|^2 = \beta\|T_\Lambda\vphi\|_2^2\,\le\,\beta\|T_\Lambda\|^2\|\vphi\|_{H^2(\D)}^2.
\end{equation*}
Hence, the condition (ii) in Lemma \ref{l:bessel} is satisfied so that $(A^nf_i)_{n\in\N}$ indeed is a Bessel sequence. Let $h\in\calH$ be arbitrary. By (iii) and (iv) (see also Remark \ref{r:jojo}(c)), there exists $\phi\in H_I^2$, $\phi = (\vphi_i)_{i\in I}$, such that $C_E^*T_{\Lambda,I}\phi = h$. For $i\in I$, let $(c_{in})_{n\in\N}\in\ell^2(\N)$ such that $\vphi_i(z) = \sum_{n=0}^\infty c_{in}z^n$ for $z\in\D$. Then
$$
\sum_{i\in I}\sum_{n=0}^\infty c_{in}A^nf_i = \sum_{i\in I}\sum_{j=0}^\infty\sum_{n=0}^\infty c_{in}\la_j^nP_jf_i = \sum_{i\in I}\sum_{j=0}^\infty\veps_j\vphi_i(\la_j)(\veps_j^{-1}P_jf_i) = h.
$$
Hence, the synthesis operator of $\calA = (A^nf_i)_{n\in\N,\,i\in I}$ is onto, which means that $\calA$ is a frame for $\calH$.
\end{proof}

From Theorem \ref{t:characterization} we deduce the following two corollaries.

\begin{cor}\label{c:charac_us}
Let $A = \sum_{j=0}^\infty\la_jP_j$ be a diagonal operator in normal form such that $(\la_j)_{j\in\N}\subset\D$ is uniformly separated and let $(f_i)_{i\in I}$ be a finite sequence of vectors in $\calH$. Then $\calA$ is a frame for $\calH$ if and only if there exist $\alpha,\beta > 0$ such that
\begin{equation*}
\alpha(1-|\la_j|^2)\|h\|^2\,\le\,\sum_{i\in I}|\<h,f_i\>|^2\,\le\,\beta(1-|\la_j|^2)\|h\|^2,\quad j\in\N,\,h\in P_j\calH.
\end{equation*}
\end{cor}
\begin{proof}
Since $(\la_j)_{j\in\N}$ is assumed to be uniformly separated, the conditions (ii) and (iv) from Theorem \ref{t:characterization} are trivially satisfied (see Theorem \ref{t:interpolating}). Thus, $\calA$ is a frame for $\calH$ if and only if condition (iii) from Theorem \ref{t:characterization} holds.
\end{proof}

\begin{cor}\label{c:nownew}
Assume that $A = \sum_{j=0}^\infty\la_jP_j$ is a diagonal operator in normal form and that $\dim P_j\calH = |I| < \infty$ for all but a finite number of $j\in\N$. Then $\calA$ is a frame for $\calH$ if and only if the following two statements hold:
\begin{enumerate}
\item[{\rm (a)}] $(\la_j)_{j\in\N}$ is a uniformly separated sequence in $\D$.
\item[{\rm (b)}] There exist $\alpha,\beta > 0$ such that for all $j\in\N$ and all $h\in P_j\calH$ we have
$$
\alpha(1-|\la_j|^2)\|h\|^2\,\le\,\sum_{i\in I}|\<h,f_i\>|^2\,\le\,\beta(1-|\la_j|^2)\|h\|^2.
$$
\end{enumerate}
\end{cor}
\begin{proof}
If (a) and (b) hold, then $\calA$ is a frame for $\calH$ by Corollary \ref{c:charac_us}. Conversely, let $\calA$ be a frame for $\calH$. Then (b) follows from Theorem \ref{t:characterization}. For the proof of (a) let us first assume that $\dim P_j\calH = |I|$ for {\it all} $j\in\N$. Since $(\veps_j^{-1}P_jf_i)_{i\in I}$ is a frame for $P_j\calH$ for every $j\in\N$ with frame bounds $\alpha$ and $\beta$ (see Remark \ref{r:jojo}), we conclude that it is even a Riesz basis of $P_j\calH$ with Riesz bounds $\alpha$ and $\beta$. Hence, $E = (\veps_j^{-1}P_jf_i)_{j\in\N,\,i\in I}$ is a Riesz basis of $\calH$. Thus, the synthesis operator $C_E^*$ is one-to-one, i.e., $\ker C_E^* = \{0\}$. Therefore, it is a consequence of Theorem \ref{t:characterization} that $T_\Lambda : H^2(\D)\to\ell^2(\N)$ is onto which is equivalent to $(\la_j)_{j\in\N}$ being uniformly separated (see Theorem \ref{t:interpolating}).

For the general case, let $J := \{j\in\N : \dim P_j\calH = |I|\}$ and let $P := \sum_{j\in J}P_j$. Then $A_0 := A|P\calH = \sum_{j\in J}\la_j(P_j|P\calH)$ and $(A_0^nPf_i)_{n\in\N,\,i\in I}$ is a frame for $P\calH$. Thus, $(\la_j)_{j\in J}$ is uniformly separated by the first part of the proof. The claim now follows from Corollary \ref{c:finunifsep}.
\end{proof}

\begin{rem}\label{r:counterex}
For every given finite index set $I$ there exist a diagonal operator $A$ and vectors $f_i$, $i\in I$, such that $(A^nf_i)_{n\in\N,\,i\in I}$ is a frame for $\calH$ and $\ind(\Lambda) = |I|$ (where $\Lambda$ is the sequence of the distinct eigenvalues of $A$). As an example, let $m:=|I|$ and let $\Lambda = (\la_j)_{j\in\N}$ be a union of $m$ uniformly separated subsequences $\Lambda_i$, $i=1,\ldots,m$, such that $\la_j\neq\la_k$ for $j\neq k$ and $\ind(\Lambda) = m$. Note that each $\Lambda_i$ is necessarily infinite by Corollary \ref{c:finunifsep}. For $i=1,\ldots,m$ let $\Lambda_i = (\la_{j,i})_{j\in\N}$ and put $A_i := \sum_{j=0}^\infty\la_{j,i}\<\cdot,e_j\>e_j$, where $(e_j)_{j\in\N}$ is an orthonormal basis of a Hilbert space $H$. Choose vectors $g_1,\ldots,g_m\in H$ such that $(A_i^ng_i)_{n\in\N}$ is a frame for $H$, $i=1,\ldots,m$. This is possible by Theorem \ref{t:one}. Now, put $\calH := \bigoplus_{i=1}^m H$ and $A := \bigoplus_{i=1}^m A_i$, and let $f_i$ be the vector in $\calH$ with $(f_i)_i = g_i$ and $(f_i)_k = 0$ for $k\in\{1,\ldots,m\}\setminus\{i\}$. Then $(A^nf_i)_{n\in\N,\,i\in I}$ is a frame for $\calH$, since for $h = h_1\oplus\ldots\oplus h_m\in\calH$ we have
\begin{align*}
\sum_{i=1}^m\sum_{n=0}^\infty\left|\left\<h,A^nf_i\right\>\right|^2 = \sum_{i=1}^m\sum_{n=0}^\infty\left|\left\<h_i,A_i^ng_i\right\>\right|^2,
\end{align*}
from where it is easily seen that the claim is true.
\end{rem}

\vspace{.2cm}
\section{Dynamical Sampling with general bounded operators}\label{s:general}
In this section, we drop the requirement that $A\in L(\calH)$ be normal. Similarly as before, we fix $A\in L(\calH)$, an at most countable index set $I$, and vectors $f_i\in\calH$, $i\in I$, and define
\begin{equation}\label{e:calA2}
\calA := \calA(A,(f_i)_{i\in I}) := (A^nf_i)_{n\in\N,\,i\in I}.
\end{equation}
The {\em spectrum} of an operator $T\in L(\calH)$ (i.e., the set of all $\la\in\C$ for which $T - \la := T - \la\Id$ is not boundedly invertible) is denoted by $\sigma(T)$, the set of all eigenvalues of $T$ (usually called the {\em point spectrum} of $T$) by $\sigma_p(T)$. The continuous spectrum of $T$ is the set of all $\la\in\sigma(T)\setminus\sigma_p(T)$ for which $\ran(T - \la)$ is dense in $\calH$. It is denoted by $\sigma_c(T)$. The spectral radius of $T$ will be denoted by $r(T)$, i.e., $r(T) := \sup\{|\la| : \la\in\sigma(T)\}$.

Recall that an operator $T\in L(\calH)$ is said to be {\em strongly stable} if $T^nf\to 0$ as $n\to\infty$ holds for each $f\in\calH$. In this case, it follows from the Banach-Steinhaus theorem that $T$ is {\em power-bounded}, i.e., $\sup_{n\in\N}\|T^n\| < \infty$. Consequently, we infer from Gelfand's formula for the spectral radius that $r(T) = \lim_{n\to\infty}\|T^n\|^{1/n}\,\le\,1$. Hence, the spectrum of a strongly stable operator $T$ is contained in the closed unit disk. It is, moreover, quite easily shown that neither $T$ nor $T^*$ can have eigenvalues on the unit circle $\T$. The first statement of the following lemma was proved in \cite{ap}.

\begin{lem}\label{l:armenak}
If $\calA$ is a frame for $\calH$, then $A^*$ is strongly stable. In particular, we have
$$
\sigma(A)\subset\ol\D,\qquad \sigma(A)\cap\T\,\subset\,\sigma_c(A)
\qquad\text{and}\qquad
\sigma(A^*)\cap\T\,\subset\,\sigma_c(A^*).
$$
\end{lem}

The next theorem completes the necessary condition from Lemma \ref{l:armenak} to a characterizing set of conditions.

\begin{thm}\label{t:charac}
The system $\calA = (A^nf_i)_{n\in\N,\,i\in I}$ is a frame for $\calH$ if and only if the following conditions are satisfied.
\begin{enumerate}
\item[{\rm (i)}]   $A^*$ is strongly stable.
\item[{\rm (ii)}]  $(f_i)_{i\in I}$ is a Bessel sequence.
\item[{\rm (iii)}] There exists a boundedly invertible operator $S\in L(\calH)$ such that
\begin{equation}\label{e:identity}
ASA^* = S - F,
\end{equation}
where $F$ is the frame operator of $(f_i)_{i\in I}$.
\end{enumerate}
If the conditions {\rm (i)--(iii)} are satisfied, then the operator $S$ in {\rm (iii)} is necessarily the frame operator of $\calA$.
\end{thm}
\begin{proof}
Assume that $\calA$ is a frame for $\calH$ and let $S$ be its frame operator. As a subsequence of the frame $\calA$, $(f_i)_{i\in I}$ is a Bessel sequence. Furthermore, for $f\in\calH$ we have
\begin{align*}
ASA^*f
&= \sum_{i\in I}\sum_{n=0}^\infty\<A^*f,A^nf_i\>A^{n+1}f_i = \sum_{i\in I}\sum_{n=1}^\infty\<f,A^nf_i\>A^nf_i\\
&= \sum_{i\in I}\sum_{n=0}^\infty\<f,A^nf_i\>A^nf_i - \sum_{i\in I}\<f,f_i\>f_i = Sf - Ff,
\end{align*}
which proves \eqref{e:identity}. For the converse statement, assume that (i)--(iii) are satisfied. From (iii) (and (ii)) we conclude that
$$
A^2S(A^*)^2 = A(ASA^*)A^* = A(S-F)A^* = ASA^* - AFA^* = S - (F + AFA^*).
$$
For $n\in\N$, $n\ge 1$, we obtain by induction
$$
A^nS(A^*)^n = S - \sum_{k=0}^{n-1}A^kF(A^*)^k = S - \sum_{k=0}^{n-1}\sum_{i\in I}\left\<\cdot\,,A^kf_i\right\>A^kf_i.
$$
Hence, for $f\in\calH$ we have (using (i))
$$
\sum_{k=0}^{\infty}\sum_{i\in I}\left|\left\<f,A^kf_i\right\>\right|^2 = \<Sf,f\> - \lim_{n\to\infty}\left\<S(A^*)^nf,(A^*)^nf\right\> = \<Sf,f\>.
$$
Therefore, $S$ is selfadjoint and non-negative. The claim now follows from the fact that $S$ is boundedly invertible. Moreover, if $S_0$ denotes the frame operator of $\calA$, then $\<(S - S_0)f,f\> = 0$ for $f\in\calH$, which implies $S = S_0$.
\end{proof}

The next theorem describes the ``admissible'' operators $A\in L(\calH)$ for which there exists a (finite or infinite) sequence $(f_i)_{i\in I}$ such that $\calA$ becomes a frame for $\calH$. It shows that these are similar to certain contractions.

\begin{thm}\label{t:admissible}
For $A\in L(\calH)$ the following statements hold.
\begin{enumerate}
\item[{\rm (i)}]  There exists a Bessel family $\{f_i : i\in I\}\subset\calH$ such that $\calA$ in \eqref{e:calA2} is a frame for $\calH$ if and only if $A^*$ is similar to a strongly stable contraction.
\item[{\rm (ii)}] There exists a finite set $\{f_i : i\in I\}\subset\calH$ such that $\calA$ in \eqref{e:calA2} is a frame for $\calH$ if and only if $A^*$ is similar to a strongly stable contraction $T\in L(\calH)$ such that $\ran(\Id - T^*T)$ is finite-dimensional.
\end{enumerate}
If the conditions in {\rm (i)} or {\rm (ii)} are satisfied, then a contraction as in {\rm (i)} of {\rm (ii)} is given by
$$
T = S^{1/2}A^*S^{-1/2},
$$
where $S$ is the frame operator of $\calA$.
\end{thm}
\begin{proof}
Assume that there exists a Bessel family $\calF = \{f_i : i\in I\}\subset\calH$ such that $\calA$ is a frame for $\calH$. By Theorem \ref{t:charac}, $A^*$ is strongly stable and $ASA^* = S - F$, where $S$ and $F$ are the frame operators of $\calA$ and $\calF$, respectively. Define $T := S^{1/2}A^*S^{-1/2}$. Then $T$ is strongly stable and $T^*T = \Id - F'$, where $F' = S^{-1/2}FS^{-1/2}$. Note that $F'$ is a non-negative selfadjoint operator (since $F$ is). Therefore, for $f\in\calH$ we have
$$
\|Tf\|^2 = \<T^*Tf,f\> = \|f\|^2 - \<F'f,f\>\,\le\,\|f\|^2,
$$
which shows that $T$ is a contraction. If $\calF$ is finite, then $\ran(\Id - T^*T) = \ran F'$ is finite-dimensional.

Conversely, assume that $A^*$ is similar to a strongly stable contraction $T\in L(\calH)$. Then the operator $G := \Id - T^*T$ is selfadjoint and non-negative. Hence, if $U : \ell^2(\N)\to\calH$ is any unitary operator and we put $g_i := G^{1/2}Ue_i$ ($e_i$ being the $i$-th standard basis vector of $\ell^2(\N)$), we have that $(g_i)_{i\in\N}$ is a Bessel sequence and $G = \sum_{i\in\N}\<\,\cdot\,,g_i\>g_i$ is its frame operator. By Theorem \ref{t:charac}, $(T^{*n}g_i)_{n,i\in\N}$ is a frame for $\calH$. Hence, if $A = LT^*L^{-1}$ with a boundedly invertible $L\in L(\calH)$, then $A^nLg_i = LT^{*n}g_i$, so that $\calA$ is a frame for $\calH$ with $f_i = Lg_i$, $i\in I := \N$. If $\ran G$ is finite-dimensional, we can choose $U$ such that $Ue_i\in\ker G^{1/2}$ for $i\ge m := \dim\ran G^{1/2}$. Then $g_i = 0$ for $i\ge m$ and hence $(A^nf_i)_{n\in\N,\,i\in I}$ is a frame for $\calH$, where $I := \{0,\ldots,m-1\}$.
\end{proof}

\begin{rem}
The operators that are similar to a contraction have been found by V.I.\ Paulsen in \cite[Cor.\ 3.5]{pa} to be exactly those operators which are {\it completely polynomially bounded}. For the definition of this term and more details we refer to \cite{pa}.
\end{rem}

In what follows, we will mainly focus on the situation in which $I$ is a finite index set -- or at least the frame operator of $(f_i)_{i\in I}$ is a compact operator. In this case, it is clear that $(f_i)_{i\in I}$ itself cannot be a frame for $\calH$ unless $\dim\calH < \infty$. The next proposition is key to most of our observations below. For the notion {\em semi-Fredholm} and the corresponding results used below we refer the reader to Appendix \ref{s:semi-fredholm}. Recall that an operator $T\in L(\calH)$ is said to be {\em finite-dimensional} or of {\em finite-rank} if $\dim\ran T < \infty$.

\begin{prop}\label{p:fredholm}
Assume that $\calA$ is a frame for $\calH$. If the frame operator $F$ of $(f_i)_{i\in I}$ is compact, then for each $\la\in\D$ the operator $A^* - \la$ is upper semi-Fredholm. If $|I|$ is finite \braces{in which case $F$ is even finite-dimensional\,}, then
\begin{equation}\label{e:fd}
\nul(A^* - \la)\,\le\,|I|,\qquad\la\in\D.
\end{equation}
\end{prop}
\begin{proof}
We derive the claim from the identity \eqref{e:identity}. For $\la\in\D$ we have
$$
AS(A^*-\la) = ASA^* - \la AS = S - F - \la AS = (\Id - \la A)S - F.
$$
For all $\la\in\D$ the operator $B_\la := \Id - \la A$ is boundedly invertible. This is clear for $\la = 0$, and for $\la\neq 0$ we have $B_\la = \la(\la^{-1} - A)$, which is boundedly invertible as $\sigma(A)\subset\ol\D$. Thus,
\begin{equation}\label{e:nice}
B_\la^{-1}AS(A^* - \la) = S - B_\la^{-1}F.
\end{equation}
By Theorem \ref{t:compi}, the operator on the right hand side is Fredholm, and so $A^* - \la$ is upper semi-Fredholm by Lemma \ref{l:komisch}.

Now, let $|I|$ be finite and let $\la\in\D$ be an eigenvalue of $A^*$. If $f$ is a corresponding eigenvector, then \eqref{e:nice} yields $f = S^{-1}B_\la^{-1}Ff$. Hence, $\ker(A^*-\la)\subset S^{-1}B_\la^{-1}\ran F$, which implies \eqref{e:fd} as $\dim\ran F\le |I|$.
\end{proof}

In the proof of the next theorem we heavily make use of the punctured neighborhood theorem, Theorem \ref{t:pnt}.

\begin{thm}\label{t:alt}
If $\calA$ is a frame for $\calH$ and the frame operator of $(f_i)_{i\in I}$ is compact, then $\ind(A^*-\la) = \ind(A^*)$ for each $\la\in\D$ and exactly one of the following cases holds:
\begin{enumerate}
\item[{\bf (i)}]   $\sigma(A^*) = \ol\D$ and $\sigma_p(A^*) = \D$.
\item[{\bf (ii)}]  $\sigma(A^*) = \ol\D$ and $\sigma_p(A^*)$ is discrete in $\D$.
\item[{\bf (iii)}] $\sigma(A^*)$ is discrete in $\D$.
\end{enumerate}
In the case {\bf (i)}, each $\la\in\D$ is an eigenvalue of $A^*$ with infinite algebraic multiplicity, whereas in the cases {\bf (ii)} and {\bf (iii)} the eigenvalues of $A^*$ in $\D$ have finite algebraic multiplicities. If case {\bf (iii)} holds, then we have $\ind(A^*) = 0$.
\end{thm}
\begin{proof}
By Proposition \ref{p:fredholm}, $A^*-\la$ is upper semi-Fredholm for each $\la\in\D$. Hence it follows from the punctured neighborhood theorem, Theorem \ref{t:pnt}, and a compactness argument that $\ind(A^*-\la)$ is constant on $\D$. Similarly, one sees that $\nul(A^*-\la)$ is constant on $\D\setminus\Delta$, where $\Delta$ is a discrete subset of $\D$. Denote this constant value by $n_0$. Then it is immediate that case {\bf (i)} is satisfied exactly if $n_0 > 0$. If $n_0 = 0$, then case {\bf (iii)} occurs if and only if $\ind(A^*) = 0$.

Let $\la_0\in\D$ be an eigenvalue of $A^*$. For $\la\neq\la_0$ close to $\la_0$ we have $\la\in\D\setminus\Delta$ and hence $n_0 = \nul(A^*-\la) = \nul(A^* - \la_0) - k$, where (see Theorem \ref{t:pnt})
$$
k = \dim\big(\ker(A^*-\la_0)/\left(\ker(A^* - \la_0)\cap R_\infty(A^*-\la_0)\right)\big).
$$
Hence, cases {\bf (ii)} and {\bf (iii)} occur exactly when $\nul(A^*-\la_0) = k$. This happens if and only if $\ker(A^* - \la_0)\cap R_\infty(A^*-\la_0) = \{0\}$. But the latter means that the algebraic multiplicity of $\la_0$ as an eigenvalue of $A^*$ is finite.
\end{proof}

\begin{cor}\label{c:r=1}
If $\dim\calH = \infty$, $\calA$ is a frame for $\calH$ with frame operator $S$, and the frame operator of $(f_i)_{i\in I}$ is compact, then
$$
r(A) = \big\|S^{-1/2}AS^{1/2}\big\| = 1.
$$
In particular, $\|A^n\|\ge 1$ for all $n\in\N$.
\end{cor}
\begin{proof}
It follows from Theorem \ref{t:alt} that $r(A) = 1$. By Theorem \ref{t:admissible}, the operator $B := S^{-1/2}AS^{1/2}$ is a contraction. Since $B$ is similar to $A$, we have $\sigma(B) = \sigma(A)$ and therefore $1 = r(A) = r(B)\le\|B\|\le 1$.
\end{proof}

We define the {\em essential spectrum} $\sess(T)$ of $T\in L(\calH)$ by those $\la\in\C$ for which $T - \la$ is not semi-Fredholm.

\begin{cor}
Assume that $\calA$ is a frame for $\calH$ and the frame operator of $(f_i)_{i\in I}$ is compact. Then
$$
\sess(A^*) = \sigma_c(A^*) = \sigma(A^*)\cap\T.
$$
If, in addition, $\ind(A^*)\neq 0$, then
$$
\sess(A^*) = \sigma_c(A^*) = \T.
$$
\end{cor}
\begin{proof}
$\sigma_c(A^*)\subset\sess(A^*)$ holds by definition and $\sess(A^*)\subset\sigma(A^*)\cap\T$ is a direct consequence of Proposition \ref{p:fredholm}. The remaining inclusion $\sigma(A^*)\cap\T\subset\sigma_c(A^*)$ holds due to Lemma \ref{l:armenak}. If $\ind(A^*)\neq 0$, then either case {\bf (i)} or case {\bf (ii)} holds. In these cases, we have $\sigma(A^*) = \ol\D$ and hence, clearly, $\sigma(A^*)\cap\T = \T$.
\end{proof}

In the proof of Theorem \ref{t:alt} we have not used that the operator $A^*$ is strongly stable. We incorporate this in the proof of the next theorem, where we make use of a theorem from \cite{tu}.

\begin{thm}\label{t:tu}
Let $I$ be finite and assume that $\calA$ is a frame for $\calH$. Then $\defi(A^*-\la) = 0$ for all $\la\in\D\setminus\Delta$, where $\Delta\subset\D$ is discrete in $\D$. In particular, either case {\bf (i)} or case {\bf (iii)} occurs. Case {\bf (iii)} holds if and only if $\ind(A^*) = 0$. In this case, also $A$ is strongly stable.
\end{thm}
\begin{proof}
Let $S$ and $F$ be the frame operators of $\calA$ and $(f_i)_{i\in I}$, respectively, and define $T := S^{1/2}A^*S^{-1/2}$. By Theorem \ref{t:admissible}, the operator $T$ is a strongly stable contraction and $T^*T = \Id - F_1$, where $F_1 = S^{-1/2}FS^{-1/2}$. Since $T$ and $A^*$ are similar, it suffices to prove the corresponding statements for the operator $T$.

Let us show the first part of the theorem. To this end, we shall use techniques from the proof of \cite[Lemma 1.3]{uch}. The key in this proof is a triangulation of the contraction $T$ of the form (see \cite[Theorem II.4.1]{nfbk})
$$
T = \mat{T_{01}}C0{T_{00}}
$$
with respect to a decomposition $\calH = \calH_{01}\oplus\calH_{00}$. Here, $T_{01}\in C_{01}$ (that is, $\inf\{\|(T_{01}^*)^nf\| : n\in\N\} > 0$ for each $f\in\calH_{01}\setminus\{0\}$) and $T_{00}\in C_{00}$ (i.e., both $T_{00}$ and $T_{00}^*$ are strongly stable). We have
$$
\Id - T^*T = \Id - \mat{T_{01}^*}0{C^*}{T_{00}^*}\mat{T_{01}}C0{T_{00}} = \mat{\Id - T_{01}^*T_{01}}{-T_{01}^*C}{-C^*T_{01}}{\Id - C^*C - T_{00}^*T_{00}}.
$$
Hence, all entries in the latter operator matrix are of finite rank. In particular, $T_{01}$ is upper semi-Fredholm (see Theorem \ref{t:compi} and Lemma \ref{l:komisch}). Since $T_{01}\in C_{01}$, the operator $T_{01}^*$ is injective and thus it has a bounded left-inverse. Hence, as $T_{01}^*C$ is of finite rank we infer that also $C$ is of finite rank. Thus, $\Id - T_{00}^*T_{00}$ is of finite rank. Since $T_{00}\in C_{00}$, this yields that $T_{00}$ is a so-called $C_0$-contraction (see \cite{tu}). Consequently, the spectrum of $T_{00}$ in $\D$ is discrete (cf.\ \cite[Theorem III.5.1]{nfbk}).

Let $\la\in\D\setminus\sigma(T_{00})$. Then $(T^*-\ol\la)f = 0$ implies $(T_{01}^*-\ol\la)g = 0$ and $C^*g + (T_{00}^*-\ol\la)h = 0$, where $f = g+h$, $g\in\calH_{01}$, $h\in\calH_{00}$. But as $T_{01}^*-\ol\la$ is injective (due to $T_{01}\in C_{01}$), we conclude that $g = 0$ and therefore also $h = 0$ as $\ol\la\in\rho(T_{00}^*)$. Hence, for $\la\in\D\setminus\sigma(T_{00})$ we have that $\defi(T-\la) = \nul(T^*-\ol\la) = 0$. This also implies that case {\bf (ii)} cannot occur and that case {\bf (iii)} holds if and only if $\ind(T) = 0$.

Assume that $\ind(T) = 0$. In order to show that also $T^*$ is strongly stable, due to \cite[Theorem 2]{tu} it suffices to prove that $\Id - TT^*$ is of finite rank. To see this, we observe that there exists a representation $T = U|T|$ of $T$, where $|T|= (T^*T)^{1/2}$ and $U$ is a unitary operator in $\calH$, see \cite[Lemma 2.9]{dp}. We have $F_1 = \Id - T^*T = \Id - |T|^2$ and thus $|T| = \Id - F_1(\Id+|T|)^{-1}$. Therefore, $T = U|T| = U - F_2$ with some finite rank operator $F_2$, and consequently
$$
TT^* = \Id - \left[F_2U^* + (U - F_2)F_2^*\right].
$$
Thus, $\Id - TT^*$ is of finite rank.
\end{proof}

\begin{rem}
It follows from the proof of Theorem \ref{t:tu} and the references used therein that the claim of the theorem remains to hold if we replace the condition that $I$ be finite by the requirement that the frame operator of $(f_i)_{i\in I}$ is of trace class.
\end{rem}

By $R$ and $L$ we denote the right-shift and the left-shift on $\ell^2(\N)$. That is, $R,L\in L(\ell^2(\N))$, $(Lf)(j) = f(j+1)$ for $j\in\N$ and $(Rf)(0) = 0$, as well as $(Rf)(j) = f(j-1)$ for $j\ge 1$. Moreover, let $e_k$ denote the $k$-th standard basis vector of $\ell^2(\N)$, $k\in\N$. The following example shows in particular that case {\bf (i)} in Theorem \ref{t:alt} cannot be neglected as a possibility for an operator generating a frame.

\begin{ex}\label{ex:shift}
Let $\calH = \ell^2(\N)$, $m\in\N\setminus\{0\}$, and $I := \{0,\ldots,m-1\}$. If we put $A := R^m$, then we have
$$
\calA = \left(A^ne_i\right)_{n\in\N,\,i\in I} = \left((R^m)^ne_i\right)_{n\in\N,\,i\in I} = (e_{i + nm})_{n\in\N,\,i\in I} = (e_k)_{k\in\N}.
$$
Hence, $(A^ne_i)_{n\in\N,\,i\in I}$ is an orthonormal basis of $\calH$ and it is easily seen that every $\la\in\D$ is an eigenvalue of $A^*$. Thus, we are in the situation of case {\bf (i)}.
\end{ex}

The next theorem shows that the orthonormal bases in Example \ref{ex:shift} are the prototype of all Riesz bases of the form $\calA$ in the sense of the following theorem.

\begin{thm}\label{t:riesz_sequence}
Let $A\in L(\calH)$ and $f_i\in\calH$, $i\in I$, where $I = \{0,\ldots,m-1\}$. Then the following statements are equivalent.
\begin{enumerate}
\item[{\rm (i)}]  $\calA$ is a Riesz basis of $\calH$.
\item[{\rm (ii)}] There exists a boundedly invertible operator $V\in L(\ell^2(\N),\calH)$ such that
$$
A = VR^mV^{-1}
\qquad\text{and}\qquad
f_i = Ve_i,\;i\in I.
$$
\end{enumerate}
\end{thm}
\begin{proof}
It is clear that (ii) implies (i). So, assume that $(A^nf_i)_{n\in\N,\,i\in I}$ is a Riesz basis of $\calH$. Then there exists a boundedly invertible operator $V\in L(\ell^2(\N),\calH)$ such that $A^nf_i = Ve_{i+nm}$, $n\in\N$, $i\in I$. In particular, for $i\in I$ we have $f_i = Ve_i$. Also, for $n\in\N$ and $i\in I$ we have
$$
AVe_{i+nm} = A^{n+1}f_i = Ve_{i + nm + m} = VR^me_{i+nm},
$$
and therefore $AV = VR^m$.
\end{proof}

Finally, we turn back to the motivation of Dynamical Sampling in the Introduction, where $A^*$ was an instance of an operator semigroup. Recall that a semigroup of operators is a collection $(T_t)_{t\ge 0}\subset L(\calH)$ satisfying $T_{s+t} = T_sT_t$ for all $s,t\in [0,\infty)$.

\begin{prop}\label{p:not(i)}
Let $A^* = T_{t_0}$, $t_0 > 0$, be an instance of a semigroup $(T_t)_{t\ge 0}$ of operators and let the frame operator of $(f_i)_{i\in I}$ be compact. Then if $\calA$ is a frame for $\calH$, either case {\bf (ii)} or case {\bf (iii)} occurs. In case {\bf (ii)} we have $\ind(A^*) = -\infty$.
\end{prop}
\begin{proof}
Let $B_m := T_{t_0/m}^*$, $m\in\N$, $m\ge 1$. Then we have $B_m^m = A$ for each $m$. Let $\la\in\D$ be arbitrary. Then $A^*-\la^m = (B_m^*)^m - \la^m = P(B_m^*,\la)(B_m^*-\la)$, where $P(B_m^*,\la)$ is a polynomial in $B_m^*$ and $\la$. This and Lemma \ref{l:komisch} imply that $B_m^* - \la$ is upper semi-Fredholm. Moreover, by the index formula \eqref{e:indexformel} we have that $\ind(A^*) = m\cdot\ind(B_m^*)$. In particular, $\ind(A^*)$ is divisible by each $m\in\N$, $m\ge 2$. Thus, $\ind(A^*)\in\{0,-\infty\}$. Note that $\ind(A^*) = +\infty$ is not possible since $\nul(A^*) < \infty$.

Now, let $\la\in\D\setminus\{0\}$, $\la = re^{it}$, $r\in (0,1)$, $t\in [0,2\pi)$, let $n = \nul(A^*-\la)+1$, and put $\la_k := \sqrt[n]{r}\exp(\frac{i}{n}(t+2k\pi))$, $k=0,\ldots,n-1$. Suppose that each $\la_k$ is an eigenvalue of $B_n^*$ with eigenvector $g_k$, $k=0,\ldots,n-1$. Then each $g_k$ is an eigenvector of $A^*$ with respect to $\la$ as $A^*g_k = (B_n^*)^ng_k = \la_k^ng_k = \la g_k$. But as the $g_k$ are linearly independent, this is a contradiction to the choice of $n$. Thus, the eigenvalues of $B_n^*$ do not fill the open unit disk. In turn, $\sigma_p(B_n^*)$ is discrete in $\D$ and hence the same holds for $\sigma_p(A^*)$.
\end{proof}

Note that we have not used any continuity properties of the semigroup in the proof above. In fact, if $(T_t)_{t\ge 0}$ is a strongly continuous semigroup, it can be shown that under the conditions of Proposition \ref{p:not(i)}, $\ker(T_t) = \{0\}$ for each $t\ge 0$, which in particular excludes case {\bf (i)}. We conclude this section with the following corollary, which directly follows from Proposition \ref{p:not(i)}, Theorem \ref{t:riesz_sequence}, and Theorem \ref{t:tu}.

\begin{cor}\label{c:neverriesz}
Let $A^* = T_{t_0}$, $t_0 > 0$, be an instance of a semigroup $(T_t)_{t\ge 0}$ of operators. Then, for any finite sequence $(f_i)_{i\in I}$ of vectors in $\calH$, the system $(A^nf_i)_{n\in\N,\,i\in I}$ is never a Riesz basis of $\calH$. Moreover, if $(A^nf_i)_{n\in\N,\,i\in I}$ is a frame for $\calH$ and $I$ is finite, then case {\bf\rm (iii)} holds and both $A$ and $A^*$ are strongly stable.
\end{cor}

\vspace{.2cm}
\section{Dynamical Sampling in finite dimensions}\label{s:findim} 
In this section we let $\calH = \calH_d$ be a $d$-dimensional Hilbert space and consider the question for which linear operators $A\in L(\calH_d)$ and which sets of vectors $\{f_i : i\in I\}\subset\calH_d$ the system
\begin{equation}\label{e:calA_findim}
\calA := (A^nf_i)_{n\in N,\,i\in I}
\end{equation}
is a frame for $\calH_d$ (or, equivalently, complete in $\calH_d$). Here, we let
$$
N := \{0,\ldots,d-1\}
\qquad\text{and}\qquad
I := \{1,\ldots,m\},\;\;m\in\N\setminus\{0\}.
$$

The main result in this section is the following characterization theorem. Here   by $P_M$ we denote the orthogonal projection in $\calH_d$ onto the subspace $M\subset\calH_d$ and by  $\ds$ the direct sum of subspaces.

\begin{thm}\label{t:main_findim}
Let $A\in L(\calH_d)$, $f_1,\ldots,f_m\in\calH_d$, and set $\calF := \linspan\{f_1,\ldots,f_m\}$. Moreover, for each $\la\in\sigma(A)$ choose a subspace $V_\la$ such that $\calH_d = V_\la\ds\ran(A-\la)$ and denote the projection onto $V_\la$ with respect to this decomposition by $Q_{V_\la}$. Then the following statements are equivalent:
\begin{enumerate}
\item[{\rm (i)}]  The system $\calA$ in \eqref{e:calA_findim} is a frame for $\calH_d$.
\item[{\rm (ii)}] For each $\la\in\sigma(A)$ we have $Q_{V_\la}\calF = V_\la$.
\item[{\rm (iii)}] For each $\la\in\sigma(A^*)$ we have $P_{\ker(A^*-\la)}\calF = \ker(A^*-\la)$.
\end{enumerate}
\end{thm}

In the following proof we deal with root subspaces of linear operators. The {\em root subspace} of an operator $T\in L(\calH_d)$ at $\la\in\sigma(T)$ is defined by
$$
\calL_\la(T) := \bigcup_{n=0}^d \ker\left((T-\la)^n\right).
$$
It is obviously invariant under $T$. It is well known that if $\sigma(T) = \{\la_1,\ldots,\la_m\}$, then
\begin{equation}\label{e:decomp}
\calH_d = \calL_{\la_1}(T)\,\ds\,\ldots\,\ds\,\calL_{\la_m}(T).
\end{equation}

\begin{proof}[Proof of Theorem  \rmref{t:main_findim}]
(i)$\Sra$(ii). Let $\la\in\sigma(A)$ and define a scalar product $\hproduct$ on $\calH_d$ such that $V_\la$ and $\ran(A-\la)$ are $\hproduct$-orthogonal to each other. By $A^\star$ and $R_M$ we denote the adjoint of $A$ and the orthogonal projection onto a subspace $M\subset\calH_d$, respectively, both with respect to the inner product $\hproduct$. Then $V_\la = \ker(A^\star-\ol\la)$ and $Q_{V_\la} = R_{\ker(A^\star-\ol\la)}$. Now, let $f\in V_\la$ be such that $(f,Q_{V_\la}f_i) = 0$ for all $i\in I$. Then for each $i\in I$ and $n\in\{0,\ldots,d-1\}$ we have
$$
\left(f,A^{n}f_i\right) = \left(A^{\star n}f,f_i\right) = \ol\la^n\left(f,f_i\right) = \ol\la^n\left(f,Q_{V_\la}f_i\right) = 0.
$$
Hence, (i) implies $f = 0$.

(ii)$\Sra$(iii). Let $\la\in\sigma(A)$ be arbitrary. Then we have
$$
Q_{V_\la}P_{\ker(A^*-\ol\la)} = Q_{V_\la}(\Id - P_{\ran(A-\la)}) = Q_{V_\la} - Q_{V_\la}P_{\ran(A-\la)} = Q_{V_\la}.
$$
Therefore, if $Q_{V_\la}\calF = V_\la$, then $V_\la = Q_{V_\la}P_{\ker(A^*-\ol\la)}\calF$, which implies that the dimension of $P_{\ker(A^*-\ol\la)}\calF$ cannot be less than the dimension of $V_\la$. But $\dim V_\la = \dim\ker(A^*-\ol\la)$ and $P_{\ker(A^*-\ol\la)}\calF = \ker(A^*-\ol\la)$ follows.

(iii)$\Sra$(i). Towards a contradiction, suppose that there exists some $f\in\calH_d$, $f\neq 0$, such that $\<f,A^nf_i\> = 0$ for all $n=0,\ldots,d-1$ and all $i\in I$. In other words, we have that $\<q(A^*)f,f_i\> = 0$ for all $i\in I$ and each polynomial $q$ of degree at most $d-1$. By $p$ denote the minimal polynomial of $A^*$ and let $\la_1,\ldots,\la_M$ be the distinct eigenvalues of $A^*$. Then $p(\la) = (\la-\la_1)^{k_1}\dots(\la-\la_M)^{k_M}$ with some $k_j\in\N$, $j\in [M] := \{1,\ldots,M\}$. Clearly, we have $k_1+\dots+k_M\le d$. By \eqref{e:decomp} we can write $f = \sum_{j=1}^Mh_j$, where $h_j\in\calL_{\la_j}(A^*)$, $j\in [M]$. As $p(A^*) = 0$ and each $\calL_{\la_j}(A^*)$ is $A^*$-invariant, we have $(A^*-\la_j)^{k_j}h_j = 0$ for all $j\in [M]$. Since $f\neq 0$, there exists at least one $j$ for which $h_j\neq 0$ and we fix it for the rest of the proof. Let $\ell_j$ be the minimum of all $\ell\le k_j$ with $(A^*-\la_j I)^\ell h_j = 0$ and define the polynomial
$$
q(\la) := (\la-\la_j)^{\ell_j-1}\cdot\!\!\!\prod_{\ell\in [M]\setminus\{j\}}(\la-\la_\ell)^{k_\ell}.
$$
We obviously have $q(A^*)h_r = 0$ for $r\neq j$ and hence $q(A^*)f = q(A^*)h_j$. Now, let $g_j := (A^*-\la_j I)^{\ell_j-1}h_j$. Then $g_j\in \ker(A^*-\la_j)$, $g_j\neq 0$ (by the definition of $\ell_j$), and thus
$$
q(A^*)f = q(A^*)h_j = \prod_{\ell\in [M]\setminus\{j\}}(A^*-\la_\ell)^{k_\ell}g_j = c_jg_j,
$$
where $c_j = \prod_{\ell\in [M]\setminus\{j\}}(\la_j-\la_\ell)^{k_\ell}\neq 0$. Since $\deg(q)\le d-1$, we obtain for all $i=1,\ldots,m$,
$$
\<g_j,f_i\> = c_j^{-1}\<q(A^*)f,f_i\> = 0.
$$
But as $g_j\in\ker(A^*-\la_j)$ and $\{P_{\ker(A^*-\la_j)}f_i\}_{i=1}^m$ is complete in $\ker(A^*-\la_j)$ by (ii), it follows that $g_j = 0$, which is the desired contradiction.
\end{proof}

\begin{rem}
Note that in the proof of (ii)$\Sra$(iii) we actually proved that for any fixed subspace $W$ of $\calH_d$ and any pair $V,V'$ of subspaces complementary to $W$ the following holds: For each subspace $\calF$ of $\calH_d$ we have $Q_V\calF = V$ if and only if $Q_{V'}\calF = V'$.
\end{rem}

The first characterization for $\calA$ to be a frame for $\calH_d$ was proved in \cite{acmt}. To formulate it here, let us introduce the notion of subspaces of cyclic vectors. For this, let $\la\in\sigma(T)$, where $T\in L(\calH_d)$. A subspace $W_\la$ will be called a {\em subspace of cyclic vectors} for $T$ at $\la\in\sigma(T)$ if
\begin{equation}\label{e:soitis}
\calL_\la(T) = W_\la\ds(T-\la)\calL_\la(T).
\end{equation}
For such a subspace $W_\la$, we set $Q_{W_\la} := Q_\la P_\la$, where $P_\la$ is the projection onto $\calL_\la(T)$ with respect to the decomposition \ref{e:decomp} and $Q_\la$ is the projection in $\calL_\la(T)$ onto $W_\la$ with respect to \eqref{e:soitis}.

\begin{thm}[{\cite[Theorem 2.11]{acmt}}]\label{t:acmt}
Let $A\in L(\calH_d)$, $f_1,\ldots,f_m\in\calH_d$, and fix subspaces of cyclic vectors $W_\la$ for $A$, $\la\in\sigma(A)$. Then $\calA$ in \eqref{e:calA_findim} is a frame for $\calH_d$ if and only if for any $\la\in\sigma(A)$ we have $Q_{W_{\la}}\calF = W_{\la}$, where $\calF := \linspan\{f_1,\ldots,f_m\}$.
\end{thm}

Theorem \ref{t:acmt} is in fact a consequence of Theorem \ref{t:main_findim} because for each subspace $W_\la$ of cyclic vectors for $A$ at $\la$ we have $\calH_d = W_\la\ds\ran(A-\la)$ and $Q_{W_\la}$ actually is the projection onto $W_\la$ along $\ran(A-\la)$.

\vspace{.2cm}
\appendix

\vspace*{.2cm}
\section{Semi-Fredholm operators}\label{s:semi-fredholm}
An operator $T\in L(\calH)$ is said to be {\em upper semi-Fredholm}, if $\ker T$ is finite-dimensional and $\ran T$ is closed. The operator $T$ is called {\em lower semi-Fredholm}, if $\codim\ran T < \infty$ (in this case, the range of $T$ is automatically closed). $T$ is called {\em semi-Fredholm} if it is upper or lower semi-Fredholm and {\em Fredholm} if it is both upper and lower semi-Fredholm. In all cases, one defines the {\em nullity} and {\em deficiency} of $T$ by
$$
\nul(T) := \dim\ker T\qquad\text{and}\qquad\defi(T) := \codim\ran T.
$$
The {\em index} of $T$ is defined by
$$
\ind(T) := \nul(T) - \defi(T).
$$
This value might be a positive or negative integer or $\pm\infty$. It is, moreover, easily seen that $T$ is upper semi-Fredholm if and only if $T^*$ is lower semi-Fredholm. We have $\defi(T^*) = \nul(T)$ and $\nul(T^*) = \defi(T)$ and thus $\ind(T^*) = -\ind(T)$. The next theorem shows that the semi-Fredholm property of operators is stable under compact perturbations.

\begin{thm}[{\cite[Theorem IV.5.26]{k}}]\label{t:compi}
If $K\in L(\calH)$ is compact and $T\in L(\calH)$ is upper \braces{lower\,} semi-Fredholm, then $T + K$ is upper \braces{lower, respectively\,} semi-Fredholm with $\ind(T+K) = \ind(T)$.
\end{thm}
%

For a proof of the following lemma we refer the reader to \cite[Theorems III.16.5, III.16.6, and III.16.12]{m}.

\begin{lem}\label{l:komisch}
Let $S,T\in L(\calH)$. Then the following statements hold.
\begin{enumerate}
\item[{\rm (i)}]  If $ST$ is upper semi-Fredholm, then so is $T$.
\item[{\rm (ii)}] If $S$ and $T$ are upper semi-Fredholm, then so is $ST$ and
\begin{equation}\label{e:indexformel}
\ind(ST) = \ind(S) + \ind(T).
\end{equation}
\end{enumerate}
\end{lem}

While Theorem \ref{t:compi} deals with compact perturbations, the next theorem (also known as the (extended) {\em punctured neighborhood theorem} (see \cite[Theorems 4.1 and 4.2]{ka})) is concerned with perturbations of the type $\la\Id$, where $|\la|$ is small.

\begin{thm}\label{t:pnt}
Let $T\in L(\calH)$ be upper semi-Fredholm and put
$$
k := \dim\left[\ker T/(\ker T\cap R_\infty(T))\right],
$$
where $R_\infty(T) := \bigcap_{n=0}^\infty\ran(T^n)$. Then there exists $\veps > 0$ such that for $0 < |\la| < \veps$ the following statements hold.
\begin{enumerate}
\item[{\rm (i)}]   $T-\la$ is upper semi-Fredholm.
\item[{\rm (ii)}]  $\ind(T-\la) = \ind(T)$.
\item[{\rm (iii)}] $\nul(T-\la) = \nul(T) - k$.
\item[{\rm (iv)}]  $\defi(T-\la) = \defi(T) - k$.
\end{enumerate}
\end{thm}

\vspace*{.1cm}
\section{Harmonic Analysis in the unit disk}\label{a:sequences}
In this section of the Appendix we collect some results on complex sequences in the unit disk. Recall the definition of the evaluation operator $T_\Lambda$ in \eqref{e:eval1}--\eqref{e:eval2}.

\begin{lem}\label{l:dense}
If $\one\in D(T_\Lambda)$, then $\id^n\in D(T_\Lambda)$ for all $n\in\N$. In particular, $T_\Lambda$ is densely defined. If, in addition, $\la_i\neq\la_j$ for $i\neq j$, then $\ran T_\Lambda$ is dense in $\ell^2(\N)$.
\end{lem}
\begin{proof}
First, $\one\in D(T_\Lambda)$ means that $\Lambda$ is a Blaschke sequence, i.e., $(\veps_j)_{j\in\N}\in\ell^2(\N)$. Thus, for $n\in\N$ we have that $\sum_{j=0}^\infty\veps_j^2|\la_j|^{2n}\le\sum_{j=0}^\infty\veps_j^2 < \infty$. That is, $\id^n\in D(T_\Lambda)$. If $\la_i\neq\la_j$ for $i\neq j$, for fixed $i\in\N$ let $B_i$ be the Blaschke product of $(\la_j)_{j\neq i}$, i.e.,
$$
B_i(z) = z^k\prod_{j\neq i}\frac{\la_j-z}{1 - \ol\la_jz}\frac{|\la_j|}{\la_j}\,,
$$
where $k\in\{0,1\}$ and $k=0$ iff $\la_j\neq 0$ for all $j\neq i$. Set $f_i := (\veps_iB_i(\la_i))^{-1}B_i$, $i\in\N$. Then $f_i\in H^\infty(\D)\subset H^2(\D)$ (see, e.g., \cite[Theorem 15.21]{r}). Moreover, $f_i\in D(T_\Lambda)$ and $T_\Lambda f_i$ is the $i$-th standard basis vector of $\ell^2(\N)$. Hence, $\ran T_\Lambda$ is dense in $\ell^2(\N)$.
\end{proof}

The following theorem is due to Shapiro and Shields \cite{ss}.

\begin{thm}\label{t:interpolating}
For a sequence $\Lambda = (\la_k)_{k\in\N}\subset\D$ the following statements are equivalent.
\begin{enumerate}
\item[{\rm (i)}]   The evaluation operator $T_\Lambda$ is defined on $H^2(\D)$ and is onto.
\item[{\rm (ii)}]  The sequence $(K_{\la_j})_{j\in\N}$ is a Riesz sequence in $H^2(\D)$.
\item[{\rm (iii)}] $\Lambda$ is uniformly separated.
\end{enumerate}
\end{thm}

The equivalence of (i) and (ii) in Theorem \ref{t:interpolating} simply follows from the fact that $T_\Lambda$ is the analysis operator of the sequence $(K_{\la_j})_{j\in\N}$. The following two theorems can be found in \cite{ds}.

\begin{thm}\label{t:finite_union}
Let $\Lambda = (\la_j)_{j\in\N}\subset\D$. Then the following conditions are equivalent.
\begin{enumerate}
\item[{\rm (i)}]   $\Lambda$ is a finite union of uniformly separated sequences.
\item[{\rm (ii)}]  $D(T_\Lambda) = H^2(\D)$.
\end{enumerate}
\end{thm}

\begin{thm}\label{t:durenschuster}
Let $\Lambda = (\la_j)_{j\in\N}\subset\D$. Then the following statements are equivalent.
\begin{enumerate}
\item[{\rm (i)}]  $\Lambda$ is uniformly separated.
\item[{\rm (ii)}] $\Lambda$ is separated and $D(T_\Lambda) = H^2(\D)$.
\end{enumerate}
\end{thm}

Note that Theorem \ref{t:durenschuster} is not formulated as a theorem in \cite{ds}, but is hidden in the proof of the implication (iii)$\Sra$(i) of the main theorem. It immediately implies the next corollary.

\begin{cor}\label{c:finunifsep}
Let $(\la_j)_{j\in\N}$ be a sequence in $\D$ such that $\la_j\neq\la_k$ for $j\neq k$. If $(\la_j)_{j\ge n}$ is uniformly separated, then also $(\la_j)_{j\in\N}$ is uniformly separated.
\end{cor}

\begin{lem}\label{l:index}
Let $\Lambda = (\la_j)_{j\in\N}$ be a sequence in $\D$. For $r\in (0,1)$ and $z\in\D$ define
$$
J(r,z) := \{j\in\N : \la_j\in\BB_r(z)\}
\qquad\text{and}\qquad
m_r := \sup_{z\in\D}\left|J(r,z)\right|,
$$
and assume $m_{r_0} < \infty$ for some $r_0\in (0,1)$. Then $\Lambda$ is a union of $m_{r_0}$ separated sequences \braces{or less\,} and
\begin{equation}\label{e:index}
\ind(\Lambda) = \inf\{m_r : r\in (0,r_0]\}.
\end{equation}
\end{lem}
\begin{proof}
For $J\subset\N$ define $\Lambda_J := (\la_j)_{j\in J}$ and set $M := m_{r_0}$. We will define subsets $J_1,\ldots,J_m$, $m\le M$, recursively such that $\N = \bigcup_{k=1}^m J_k$ and each $\Lambda_{J_k}$ is separated. For the definition of $J_1$, set $j_0 := 0$ and once $j_0,\ldots,j_s$ are chosen, we pick
$$
j_{s+1} := \min\left\{j > j_s : \lambda_j\notin\bigcup_{i=1}^{s}\BB_r(\lambda_{j_i})\right\}.
$$
Note that $j_{s+1}$ is well defined due to the assumption that $M < \infty$. In other words, $\lambda_{j_{s+1}}$ is the first element of the sequence $\Lambda$ that does not belong to any of the balls of radius $r$ around the previously chosen elements. Put $J_1 := \{j_s : s\in\N\}$. It is clear that $\Lambda_{J_1}$ is separated (by $r$) and that all the $\lambda_j$ that have not been chosen due to this process belong to some $\BB_r(\lambda_{j_s})$.

If $\N\setminus J_1$ is finite, we are finished. If not, proceed as before with $\N\setminus J_1$ instead of $\N$ to find an infinite set $J_2\subset\N\setminus J_1$ such that $\Lambda_{J_2}$ is separated (by $r$). Continuing in this way, the process either terminates after $m < M-1$ steps (in which case we are done) or we obtain $M-1$ separated sequences $\Lambda_{J_1},\ldots,\Lambda_{J_{M-1}}$. In this case, put
$$
J_{M} := \N\setminus\bigcup_{k=1}^{M-1}J_k.
$$
Let us prove that $\Lambda_{J_{M}}$ is separated. For this, let $j\in J_{M}$ be arbitrary. Then, as a result of the construction process, $\la_j\in\bigcup_{k=1}^{M-1}\BB_r(\la_{i_k})$, where $i_k\in J_k$, $k=1,\ldots,M-1$. Thus, we have that $j,i_1,\ldots,i_{M-1}\in J(r,\la_j)$ such that $J(r,\la_j) = \{j,i_1,\ldots,i_{M-1}\}$. Therefore, $\varrho(\la_j,\la_l)\ge r$ for all $l\in J_{M}$, $l\neq j$. Hence, $\Lambda_{J_{M}}$ is indeed separated.

It remains to prove the relation \eqref{e:index}. For this, put $m_0 := \inf\{m_r : r\in (0,r_0]\}$. It is clear that $m_0 = m_{r_1}$ for some $r_1\le r_0$ (note that $r\mapsto m_r$ is non-decreasing) and that, therefore, $n := \ind(\Lambda)\le m_0$. There exist $J_1,\ldots,J_n\subset\N$ with $\N = \bigcup_{k=1}^nJ_k$ such that $\Lambda_{J_k}$ is separated for each $k=1,\ldots,n$. Without loss of generality, let each $\Lambda_{J_k}$ be separated by $r_1$. From $m_0 = m_{r_1/2}$ we conclude that there exists some $z\in\D$ such that $|J(r_1/2,z)| = m_0$. Suppose that $n < m_0$. Then there exists some $J_k$ that contains at least two of the $m_0$ elements of $J(r_1/2,z)$, say, $j_1$ and $j_2$. But then $\varrho(\la_{j_1},\la_{j_2})\le\varrho(\la_{j_1},z) + \varrho(z,\la_{j_2}) < r_1$, contradicting the fact that $\Lambda_{J_k}$ is separated by $r_1$.
\end{proof}

We shall also make use of the following lemma which in particular shows that the map $(\D,\varrho)\to H^2(\D)$, $z\mapsto K_z$, is Lipschitz continuous. Here, $K$ is the normalized reproducing kernel of $H^2(\D)$ defined in \eqref{e:K}.

\begin{lem}\label{l:lipschitz}
For $z\in\D$ put $s_z = \sqrt{1-|z|^2}$. Then for $z,w\in\D$ the following relation holds:
$$
\|K_z - K_w\|_{H^2(\D)}^2 = (2 - s_zs_w)\varrho(z,w)^2 - \left(1 - \varrho(z,w)^2\right)\frac{(s_z-s_w)^2}{s_zs_w}.
$$
In particular,
$$
\|K_z - K_w\|_{H^2(\D)}\,\le\,\sqrt 2\,\varrho(z,w).
$$
\end{lem}
\begin{proof}
Using $1-\varrho(z,w)^2 = s_z^2s_w^2/|1-\ol zw|^2$ (see \eqref{e:hyp_identity}), we see that
\begin{align*}
\|K_z - K_w\|_{H^2(\D)}^2
&= 2 - 2\Re\<K_z,K_w\> = 2 - 2\Re\frac{s_zs_w}{1 - \ol zw}\\
&= 2 - \,2\frac{s_zs_w}{|1-\ol zw|^2}(1 - \Re(\ol zw))\\
&= 2 - \,\frac{s_zs_w}{|1-\ol zw|^2}\left(2 - 2\Re(\ol zw) - 2s_zs_w\right) - 2\left(1-\varrho(z,w)^2\right)\\
&= 2\varrho(z,w)^2 - \,\frac{s_zs_w}{|1-\ol zw|^2}\left(2 + |z-w|^2 - |z|^2 - |w|^2 - 2s_zs_w\right)\\
&= 2\varrho(z,w)^2 - \,\frac{s_zs_w}{|1-\ol zw|^2}\left((s_z-s_w)^2 + |z-w|^2\right),
\end{align*}
which proves the claim.
\end{proof}

\vspace*{.5cm}\noindent
{\bf Acknowledgements.} The authors would like to thank D. Su\'arez for sharing his knowledge on sequences in the unit disk.



\begin{thebibliography}{99}
\bibitem{aadp}
A. Aldroubi, R. Aceska, J. Davis, and A. Petrosyan,
Dynamical sampling in shift-invariant spaces,
Contemp. Math. of the AMS 603 (2013), 139--148.

\bibitem{accmp}
A. Aldroubi, C. Cabrelli, A.F. \c{C}akmak, U. Molter, and A. Petrosyan,
Iterative actions of normal operators,
J. Funct. Anal. 272 (2017), 1121--1146.

\bibitem{acmt}
A. Aldroubi, C. Cabrelli, U. Molter, and S. Tang,
Dynamical Sampling,
Appl. Comput. Harmon. Anal. 42 (2017), 378--401.

\bibitem{adk}
A. Aldroubi, J. Davis, and I. Krishtal,
Dynamical Sampling: Time Space Trade-off,
Appl. Comput. Harmon. Anal. 34 (2013), 495--503.

\bibitem{ap}
A. Aldroubi and A. Petrosyan (2017) 
Dynamical Sampling and Systems from Iterative Actions of Operators. 
In: Pesenson I., Le Gia Q., Mayeli A., Mhaskar H., Zhou DX. (eds) Frames and Other Bases in Abstract and Function Spaces. Applied and Numerical Harmonic Analysis. BirkhŠuser, Cham

\bibitem{c}
O. Christensen,
An introduction to frames and Riesz bases,
Birkh\"auser, Boston, Basel, Berlin, 2003.

\bibitem{con}
J.B. Conway,
A course in functional analysis, 2nd edition,
Springer, New York, Berlin, Heidelberg, 1990.

\bibitem{dp}
R. Duong and F. Philipp,
The effect of perturbations of linear operators on their polar decomposition,
Proc. Amer. Math. Soc. 145 (2017), 779--790.

\bibitem{ds}
P. Duren and A.P. Schuster,
Finite unions of interpolating sequences,
Proc. Amer. Math. Soc. 130 (2002), 2609--2615.

\bibitem{ka}
M.A. Kaashoek,
Stability theorems for closed linear operators,
Indag. Math. 27 (1965), 452--466.

\bibitem{k}
T. Kato,
Perturbation theory for linear operators,
2nd Edition, Springer, Berlin, Heidelberg, New York, 1980.

\bibitem{lv}
Y. Lu and M. Vetterli,
Spatial super-resolution of a diffusion field by temporal oversampling in sensor networks,
In: Acoustics, Speech and Signal Processing, 2009. IEEE International Conference on
ICASSP 2009, April 2009, pp. 2249--2252.

\bibitem{m}
V. M\"uller,
Spectral theory of linear operators and spectral systems in Banach algebras,
2nd ed., Birkh\"auser, Basel, Boston, Berlin, 2007.

\bibitem{nfbk}
B. Sz.-Nagy, C. Foia\c{s}, H. Bercovici, L. K\'erchy,
Harmonic analysis of operators on Hilbert space,
Springer, 2010.

\bibitem{pa}
V.I. Paulsen,
Every completely polynomially bounded operator is similar to a contraction,
J. Funct. Anal. 55 (1984), 1--17.

\bibitem{p}
F. Philipp,
Bessel orbits of normal operators,
J. Math. Anal. Appl. 448 (2017), 767--785.

\bibitem{rclv}
J. Ranieri, A. Chebira, Y.M. Lu, M. Vetterli,
Sampling and reconstructing diffusion fields with localized sources,
In: IEEE International Conference on Acoustics, Speech and Signal Processing
(ICASSP), 2011, May 2011, pp. 4016--4019.

\bibitem{r}
W. Rudin,
Real and complex analysis,
third edition,
McGraw-Hill Book Company, 1987.

\bibitem{ss}
H.S. Shapiro and A.L. Shields,
On some interpolation problems for analytic functions,
Amer. J. Math. 83 (1961), 513--532.

\bibitem{tu}
K. Takahashi and M. Uchiyama,
Every $C_{00}$ contraction with Hilbert-Schmidt defect operator is of class $C_0$,
J. Operator Theory 10 (1983), 331--335.

\bibitem{uch}
M. Uchiyama,
Contractions with $(\sigma,c)$ defect operators,
J. Operator Theory 12 (1984), 221--233.

\bibitem{u}
M. Unser,
Sampling -- 50 years after Shannon,
Proc. IEEE 88 (2000), 569--587.
\end{thebibliography}
\end{document}